\documentclass{amsart}
\usepackage{enumerate}
\usepackage{amssymb}
\usepackage{color}
\usepackage{mathrsfs}
\usepackage[T1,T2A]{fontenc}
\usepackage[utf8]{inputenc}

\numberwithin{equation}{section}
\newtheorem{thm}{Theorem}[section]
\newtheorem{lem}[thm]{Lemma}
\newtheorem{nota}[thm]{Notation}
\newtheorem{fact}[thm]{Fact}

\newtheorem{cor}[thm]{Corollary}

\newtheorem{ex}[thm]{Example}
\newtheorem{prob}[thm]{Problem}
\newtheorem{rem}[thm]{Remark}

\begin{document}

\title{Local invariants of non-commutative tori}

\author{Fedor Sukochev}
\address{University of New South Wales, Kensington, NSW, 2052, Australia}
\email{f.sukochev@unsw.edu.au}
\author{Dmitriy Zanin}
\address{University of New South Wales, Kensington, NSW, 2052, Australia}
\email{d.zanin@unsw.edu.au}

\begin{abstract} We consider a generic curved non-commutative torus extending the notion of conformally deformed non-commutative torus from \cite{Connes-Tretkoff}. In general, a curved non-commutative torus is no longer represented by a spectral triple, not even by a twisted spectral triple. Therefore, the geometry of this manifold is governed by a positive second order differential operator (Laplace-Betrami operator) rather than a first order differential operator (Dirac operator). For this manifold, we prove an asymptotic expansion of the heat semi-group generated by Laplace-Beltrami operator and provide an algorithm to compute the local invariants which appear as coefficients in the expansion. This allows to extend the results of \cite{Connes-Tretkoff}, \cite{Connes-Moscovici}, \cite{FaKh} (beyond conformal case and/or for multi-dimensional tori). 
\end{abstract}


\maketitle

\section{Introduction}

We begin by reviewing the classical (commutative) roots of our work, and then move to the non-commutative generalisation prompted by \cite{Connes-Tretkoff}. Finally, we explain our results for the non-commutative torus.

\subsection{Minakshisundaram-Plejel theorem and local invariants in the classical geometry}

For a $d-$dimensional Riemannian manifold $(X,g),$ there exists a natural first order differential operator $D_g$ on the space of forms called Hodge-de Rham operator. Its square $D_g^2$ is the Hodge-Laplace operator (denoted further by $\Delta_g$) and its component acting on $0$ order forms being the Laplace-Beltrami operator (also denoted by $\Delta_g$) \cite{Rosenberg}. The heat semi-group is now defined by the formula
$$t\to e^{-t\Delta_g},\quad t>0.$$
If $X$ is compact, then the resolvent of the Laplace-Beltrami operator $\Delta_g$ is compact. Hence, $e^{-t\Delta_g}$ is compact for $t>0.$ In fact, it happens that $e^{-t\Delta_g}$ belongs to the trace class for $t>0.$

In his seminal work \cite{Weyl}, Weyl proved that, for a compact manifold,
\begin{equation}\label{weyl formula}
\lim_{t\downarrow0}(4\pi t)^{\frac{d}{2}}{\rm Tr}(e^{-t\Delta_g})={\rm Vol}(X),\quad t\downarrow0.
\end{equation}
Following Weyl's work, it became an established custom to measure various geometric (and often topological) quantities associated with a Riemannian manifold $X$ in terms of its heat semi-group expansion $t\to e^{-t\Delta_g},\quad t>0.$ The mere existence of such expansion is a famous theorem of Minakshisundaram and Plejel (among all approaches to that theorem, a particularly detailed account is given in \cite{Rosenberg}; even though Theorem 3.24 there concerns only a special case $f=1,$ the proof of the formula stated below in the general case is very similar). 

Thus, for every $f\in C^{\infty}(X),$ the Minakshisundaram-Plejel theorem asserts an existence of an asymptotic expansion 
\begin{equation}\label{heat asymptotic}
{\rm Tr}(M_fe^{-t\Delta_g})\approx (4\pi t)^{-\frac{d}{2}}\cdot \sum_{\substack{k\geq0\\ k=0{\rm mod}2}}a_k(f)t^{\frac{k}{2}},\quad t\downarrow0.
\end{equation}
Here, $d$ is the dimension of $X$ and $M_f:L_2(X)\to L_2(X)$ is the operator of pointwise multiplication by $f.$ Moreover, there exist functions $A_k\in C^{\infty}(X)$ such that
\begin{equation}\label{normality of coefficients}
a_k(f)=\int_X A_k\cdot f d{\rm vol}_g,\quad k\geq0,\quad k=0{\rm mod}2,
\end{equation}
where ${\rm vol}_g$ is the standard volume element on $X$ given in local coordinates by the formula
$$d{\rm vol}_g=({\rm det}(g))^{\frac12}(x)dx.$$

Here, the summation goes over even $k$ only because the manifold is assumed not to have a boundary. For manifolds with boundary, one should also include the terms with odd $k.$

An easy computation shows that $A_0=1,$ which is consistent with \eqref{weyl formula}. Further computations (see e.g. Proposition 3.29 in \cite{Rosenberg}) show that
$$A_2=\frac16 R,$$
where $R$ is the scalar curvature of $(X,g).$ In particular, $a_2(1)$ is the Einstein-Hilbert action (see e.g. \cite{Connes-NCG}). Further, the elements $a_k,$ $k>2$ are related to local invariants of higher order \cite{Rosenberg}.

Note that $a_0$ extends to a normal state $h$ on $L_{\infty}(X)$ by the obvious formula
$$h(f)=\int_X fd{\rm vol}_g,\quad f\in L_{\infty}(X).$$
Equation \eqref{normality of coefficients} can be re-written as
$$a_k(f)=h(A_k\cdot f),\quad f\in C^{\infty}(X).$$

This paper aims to find suitable extensions of the Minakshisundaram-Plejel theorem (and, consequently, of the Weyl theorem --- see formula \eqref{weyl formula}) for non-commutative tori with generic, non-flat, metric tensor. In the spectral geometry of Riemannian manifolds, the local invariants (such as Riemannian curvature) can be {\it detected} in the asymptotic expansion of the heat semigroup with respect to the Laplace-Beltrami operator. The paradigm of Non-commutative Geometry is to {\it define} local invariants via the asymptotic expansion of a heat semi-group associated to the Laplace-Beltrami operator.

\subsection{Local invariants in the non-commutative geometry} 

This grand program began in \cite{Connes-Tretkoff} (published only in 2011, but the main concepts and techniques were developed yet in the 1990's), where special Riemannian metric (conformal deformations of a flat one) on $2-$dimensional non-commutative manifolds was considered. The authors of \cite{Connes-Tretkoff} proved that Euler characteristic of such manifold is $0$ by means of Gauss-Bonnet theorem (recall that the classical Gauss-Bonnet theorem asserts that Euler characteristic of the $2-$dimensional Riemannian manifold equals to the average of its scalar curvature). Subsequently, the scalar curvature (for the conformal deformation of the $2-$dimensional non-commutative torus) was explicitly computed in \cite{Connes-Moscovici} and \cite{FaKh1} and, later, the term $a_4$ (the first place where the Riemann curvature tensor manifests itself beyond the scalar curvature) was further computed in \cite{CoFa} (intermediate computations include about a million terms!). 

We now briefly restate the whole program as it can be surmised from \cite{Connes-Tretkoff}. Relevant definitions concerning non-commutative torus $\mathbb{T}^d_{\theta}$ are given in Subsection \ref{ncg subsection} below.
\begin{prob}\label{the problem} Let $g$ be a Riemannian metric on the non-commutative torus and let $\Delta_g$ be the Laplace-Beltrami operator.
\begin{enumerate}[{\rm (a)}]
\item\label{tpa} prove, for every $x\in C^{\infty}(\mathbb{T}^d_{\theta}),$ the existence of the asymptotic 
$${\rm Tr}(\lambda_l(x)e^{-t\Delta_g})\sim t^{-\frac{d}{2}}\sum_{\substack{k\geq0\\ k=0{\rm mod}2}}t^{\frac{k}{2}}a_k(x)\quad t\downarrow0.$$
Here, ${\rm Tr}$ denotes the classical trace on the ideal $\mathcal{L}_1(L_{2}(\mathbb{T}^d_{\theta}))$ and 
\item\label{tpb} provide explicit formulae for the functionals $x\to a_k(x),$ $k\geq0.$
\end{enumerate}
\end{prob}

\subsection{Non-commutative Riemannian geometry}\label{ncg subsection} Let $d\geq2$ and let $\theta\in M_d(\mathbb{R})$ be anti-symmetric. Let $L_{\infty}(\mathbb{T}^d_{\theta})$ be the (von Neumann algebra of a) non-commutative torus defined with the help of the matrix $\theta.$ It is represented on the Hilbert space $L_2(\mathbb{T}^d_{\theta})$ via left regular representation $\lambda_l.$ This algebra can be viewed as the weak closure of the algebra $C^{\infty}(\mathbb{T}^d_{\theta})$ (as introduced in \cite{Connes-NCG}). It is equipped with a faithful tracial state $\tau,$ which happens to be normal. All these notions and notations are fully explained in Section \ref{prelim section}.

Ha and Ponge \cite{PongeHa} presented a general notion of Riemannian metric $g$ on the non-commutative torus which includes the conformally deformed metric considered in \cite{Connes-Tretkoff} as a special case. Namely, Riemannian metric $g$ on the non-commutative torus is simply a positive element in ${\rm GL}_d(C^{\infty}(\mathbb{T}^d_{\theta}))$ (the group of invertible $d\times d$ matrices with coefficients in $C^{\infty}(\mathbb{T}^d_{\theta})$) such that the elements $g_{ij}$ and $(g^{-1})_{ij}$ are self-adjoint for all $1\leq i,j\leq d.$

A von Neumann algebra corresponding to a curved non-commutative torus is the same as for the flat non-commutative torus. It is still represented on the same Hilbert space $L_2(\mathbb{T}^d_{\theta})$ via left regular representation. The only difference between flat and non-flat Hilbert spaces is the inner product on $L_2(\mathbb{T}^d_{\theta})$ given now by the formula
$$\langle u,v\rangle_{\nu}=\tau(u^{\ast}\nu v),\quad u,v\in L_2(\mathbb{T}^d_{\theta}).$$
Here, $\nu\in C^{\infty}(\mathbb{T}^d_{\theta})$ given in formula \eqref{nu def} below should be thought of as a "square root of the determinant" of the metric tensor $g\in {\rm GL}_d(C^{\infty}(\mathbb{T}^d_{\theta})).$ 

On the Hilbert space $L_2(\mathbb{T}^d_{\theta})$ (equipped with the inner product $\langle\cdot,\cdot\rangle_{\nu}$) we define a Laplace-Beltrami operator $\Delta_g$ by setting \cite{PongeHa}
$$\Delta_g=\lambda_l(\nu^{-1})\sum_{i,j=1}^dD_i\lambda_l(\nu^{\frac12}(g^{-1})_{ij}\nu^{\frac12})D_j.$$
Here,  $\{D_i\}_{i=1}^d$ are "partial derivations" on $C^{\infty}(\mathbb{T}^d_{\theta}).$

We view this operator as a starting point for Riemannian geometry on the non-commutative torus since it dualises the notion of Riemannian metric in the same spirit as in the commutative case. 

It should be noted that the element $e^{-t\Delta_g}$ belongs (see e.g. \cite{MPSZ}) to the trace ideal $\mathcal{L}_1(L_{2}(\mathbb{T}^d_{\theta}))$ that is to the class of all bounded  operators on $L_2(\mathbb{T}^d_{\theta})$ whose singular value sequence is summable.

\subsection{Main result} Our main result stated below provides a resolution to the Problem \ref{the problem} \eqref{tpa},\eqref{tpb} above in the most general situation. 

\begin{thm}\label{main result} Let $d\geq 2$ and $0\leq g\in {\rm GL}_d(C^{\infty}(\mathbb{T}^d_{\theta}))$ be such that the elements $g_{ij}$ and $(g^{-1})_{ij}$ are self-adjoint for all $1\leq i,j\leq d.$ For every $x\in L_{\infty}(\mathbb{T}^d_{\theta}),$ there exists an asymptotic expansion
\begin{equation}\label{asymptotic expansion}
{\rm Tr}(\lambda_l(x)e^{-t\Delta_g})\sim t^{-\frac{d}{2}}\sum_{\substack{k\geq0\\ k=0{\rm mod}2}}t^{\frac{k}{2}}\tau(x\cdot \nu^{-\frac12}I_k\nu^{\frac12}),\quad t\downarrow0.
\end{equation}
Here, $I_k$ is given in Notation \ref{ik notation} and the algorithm to compute it is presented in Section \ref{defnot section}.
\end{thm}

\subsection{Connections to earlier works}

In existing literature such theorems are proved by means of pseudo-differential calculus on the non-commutative tori \cite{ConnesDG} (developed for toric manifolds in \cite{Liu1}). An alternative approach was introduced in \cite{IM,IM2} where the case of almost commutative torus was considered. The approach of \cite{IM} is based on Duhamel formula. The resulting expression in \cite{IM,IM2} for the coefficients appears to be the same as the ones in \cite{Connes-Moscovici,FaKh1,FaKh}. 

In our approach, we avoid pseudo-differential calculus or Duhamel formula replacing them with repeated resolvent identity and borrowing methods from non-commutative harmonic analysis.

The outcomes of the presented approach are of potentially wider applicability. Its main advantages are multifold:
\begin{itemize}
\item Theorem \ref{main result} holds for every $x\in L_{\infty}(\mathbb{T}^d_{\theta}),$ not just for a smooth $x;$
\item Theorem \ref{main result} holds for an arbitrary metric tensor $g\in {\rm GL}_d(C^{\infty}(\mathbb{T}^d_{\theta}))$ and not just for a conformal deformation of a flat noncommutative torus;
\item We supply the formulae for all $I_k,$ $k\geq0,$ not just for $k=0,2,4;$
\item Our approach is designed to be applicable to other important examples where pseudo-differential calculus is unavailable e.g. non-commutative spheres;
\end{itemize}

We caution the reader that Theorem \ref{main result} is {\it not} a generalisation of \cite{Connes-Moscovici} et al. In fact, in \cite{Connes-Moscovici} a version of Theorem \ref{main result} is taken as a starting point and the main focus of \cite{Connes-Moscovici,FaKh1,CoFa,Lesch} is on representing the element $I_2$ (or $I_4$) in terms of multiple operator integrals.

Computation of $I_0$ (note that the algorithm in Section \ref{defnot section} yields $I_0=\nu$) is, in fact, related to Connes Trace Theorem \cite{Connes-action} (if we ignore the fact we do not have a bona fide spectral triple). Indeed, the equality
$${\rm Tr}(\lambda_l(x)e^{-t\Delta_g})=t^{-\frac{d}{2}}\tau(x\nu)+O(t^{\frac{\epsilon-d}{2}}),\quad t\downarrow0,$$
is expected to imply (if $\Delta_g$ is replaced by $D^2$ for some Dirac-type operator $D,$ then such an implication is known to hold \cite{SZ-fubini}) that
\begin{equation}\label{trace formula}
\varphi(\lambda_l(x)(1+\Delta_g)^{-\frac{d}{2}})=c_d\tau(x\nu)
\end{equation}
for every normalised trace on $\mathcal{L}_{1,\infty}$ (the principal ideal generated by the harmonic sequence). However, a Laplace-Beltrami operator $\Delta_g$ introduced above is not a square of any Dirac-type operator (or, at least, such a Dirac-type operator $D$ is not yet constructed). Nevertheless, \eqref{trace formula} holds in full generality \cite{MPSZ}.

\subsection{Acknowledgements}

We thank Professor Connes for supplying us with "little lemma" (see Lemma \ref{cll} and Theorem \ref{sobolev connes lemma}) which radically shortened and streamlined our proof. We thank our colleagues R. Ponge (for explaining to us his approach to the Laplace-Beltrami operator in \cite{PongeHa} and for drawing our attention to \cite{Rosenberg_sigma}), B. Iochum (for interest to our work and detailed comparison with \cite{IM,IM2}), Y. Liu (for explaining to us the interplay between analytical and geometrical ideas), M. Lesch (for discussing \cite{Lesch} with us). We also thank N. Azamov, A. Ber and E. McDonald for verification of our proofs and supplying numerous suggestions which improved the exposition. 

\section{Preliminaries}\label{prelim section}

Everything in this section is folklore. We refer the reader to \cite{Rieffel} for deformation quantization (which includes non-commutative torus as a special case), to \cite{Spera} and \cite{XXY} for Sobolev spaces on the non-commutative torus and to \cite{PongeHa}, \cite{Ponge1} for various related information. 

Let $\theta\in M_d(\mathbb{R})$ be an anti-symmetric matrix. Let $A_{\theta}$ be a $\ast-$algebra generated by elements $(U_k)_{1\leq k\leq d}$ {\it of infinite order} satisfying the conditions
$$U_kU_l=e^{i\theta_{kl}}U_lU_k,\quad U_kU_k^{\ast}=U_k^{\ast}U_k=1.$$
Natural Hamel basis in $A_{\theta}$ is $(e_n)_{n\in\mathbb{Z}^d},$
$$e_n=U_1^{n_1}U_2^{n_2}\cdots U_d^{n_d},\quad n\in\mathbb{Z}^d.$$
Note that
$$e_me_n=e^{-i\sum_{j<k}\theta_{jk}n_jm_k}e_{m+n},\quad e_n^{\ast}=e^{-i\sum_{j<k}\theta_{jk}n_jn_k}e_{-n}.$$

Consider a linear functional $\tau$ on $A_{\theta}$ defined by the formula
$$\tau(e_n)=
\begin{cases}
1,& n=0\\
0,& n\neq0
\end{cases}
$$
We have (sums are finite)
$$\tau((\sum_{m\in\mathbb{Z}^d}\alpha_me_m)(\sum_{n\in\mathbb{Z}^d}\beta_n e_n))=\sum_{m,n\in\mathbb{Z}^d}\alpha_m\beta_n\tau(e_me_n)=$$
$$=\sum_{n\in\mathbb{Z}^d}e^{i\sum_{j<k}\theta_{jk}n_jn_k}\alpha_{-n}\beta_n.$$
It is now immediate that
$$\tau(xy)=\tau(yx),\quad x,y\in A_{\theta}.$$

Let us equip $A_{\theta}$ with an inner product defined by the formula
$$\langle x,y\rangle=\tau(x^{\ast}y).$$
This inner product is non-degenerate. Indeed, for $x=\sum_{n\in\mathbb{Z}^d}\alpha_ne_n,$ we have
$$\tau(x^{\ast}x)=\tau((\sum_{m\in\mathbb{Z}^d}\alpha_me_m)^{\ast}(\sum_{n\in\mathbb{Z}^d}\alpha_n e_n))=\sum_{m,n\in\mathbb{Z}^d}\overline{\alpha}_m\beta_n\tau(e_m^{\ast}e_n)=\sum_{n\in\mathbb{Z}^d}|\alpha_n|^2.$$
Hence, $\tau(x^{\ast}x)=0$ implies $x=0.$

We have that $(A_{\theta},\langle\cdot,\cdot\rangle)$ is a pre-Hilbert space. Define a Hilbert space $H$ as the completion of $(A_{\theta},\langle\cdot,\cdot\rangle).$

For $x\in A_{\theta},$ let $\lambda_l(x):A_{\theta}\to A_{\theta}$ be a linear mapping defined by the formula
$$\lambda_l(x):y\to xy,\quad y\in A_{\theta}.$$
Obviously,
$$\lambda_l(x)=\sum_{n\in\mathbb{Z}^d}\alpha_n\lambda_l(e_n),\quad x=\sum_{n\in\mathbb{Z}^d}\alpha_ne_n.$$
Note that
$$\langle\lambda_l(e_n)y,\lambda_l(e_n)y\rangle=\langle e_ny,e_ny\rangle=\tau(y^{\ast}e_n^{\ast}\cdot e_ny)=\tau(y^{\ast}y)=\langle y,y\rangle,\quad y\in A_{\theta}.$$
In particular, $\lambda_l(e_n)$ is a unitary operator on $H.$ Hence, $\lambda_l(x)$ is a bounded operator on $H$ for every $x\in A_{\theta}.$ Now, the mapping
$$x\to\lambda_l(x),\quad x\in A_{\theta}$$
is the left regular representation of the $\ast-$algebra $A_{\theta}.$ Similarly, we define the right regular representation $\lambda_r$ (even though in the present paper we only use $\lambda_r(e_n),$ $n\in\mathbb{Z}^d$).

We define $L_{\infty}(\mathbb{T}^d_{\theta})$ as the weak (or, equivalently, strong) closure of the algebra $\lambda_l(A_{\theta}).$ It is convenient to denote elements of this algebra by $\lambda_l(x).$

The state
$$A\to \langle e_0,Ae_0\rangle,\quad A\in B(H)$$
is tracial on $L_{\infty}(\mathbb{T}^d_{\theta}).$ Indeed,
$$\langle e_0,\lambda_l(x)\lambda_l(y)e_0\rangle=\langle e_0,xye_0\rangle=\tau(xy)=\tau(yx)=\langle e_0,\lambda_l(y)\lambda_l(x)e_0\rangle,\quad x,y\in A_{\theta}.$$
For $x,y\in L_{\infty}(\mathbb{T}^d_{\theta}),$ choose $x_n,y_n\in A_{\theta}$ such that
$$\lambda_l(x_n)\to\lambda_l(x),\quad \lambda_l(y_n)\to\lambda_l(y)$$
strongly as $n\to\infty.$ Hence,
$$\lambda_l(x_n)\lambda_l(y_n)\to\lambda_l(x)\lambda_l(y),\quad \lambda_l(y_n)\lambda_l(x_n)\to\lambda_l(y)\lambda_l(x)$$
strongly as $n\to\infty.$ In particular, we have
$$\langle e_0,\lambda_l(x)\lambda_l(y)e_0\rangle=\lim_{n\to\infty}\langle e_0,\lambda_l(x_n)\lambda_l(y_n)e_0\rangle=$$
$$=\lim_{n\to\infty}\langle e_0,\lambda_l(y_n)\lambda_l(x_n)e_0\rangle=\lim_{n\to\infty}\langle e_0,\lambda_l(y)\lambda_l(x)e_0\rangle.$$
Hence, our state is indeed tracial. This trace extends $\tau$ and, for this reason, is also denoted by $\tau.$

Normality of the tracial state $\tau$ follows directly from the definition. We claim that $\tau$ is a faithful trace. Indeed, if $p\in L_{\infty}(\mathbb{T}^d_{\theta})$ is a projection with $\tau(p)=0,$ then
$$\langle p(e_n),p(e_n)\rangle=\langle (p\lambda_l(e_n))(e_0),(p\lambda_l(e_n))(e_0)\rangle=$$
$$=\langle e_0,((p\lambda_l(e_n))^{\ast}(p\lambda_l(e_n)))(e_0)\rangle=\langle e_0,(\lambda_l(e_n)^{\ast}p\lambda_l(e_n))(e_0)\rangle=$$
$$=\tau(\lambda_l(e_n)^{\ast}p\lambda_l(e_n))=\tau(p\lambda_l(e_n)\lambda_l(e_n)^{\ast})=\tau(p)=0.$$
Hence, $p(e_n)=0$ for every $n\in\mathbb{Z}^d.$ Since $\{e_n\}_{n\in\mathbb{Z}^d}$ is an orthonormal basis in $H,$ it follows that $p=0.$ Hence, $\tau$ is faithful.

\begin{ex}\label{tensor example} Take $d'>d$ and consider $d'\times d'$ matrix $\theta'$ whose left upper corner is $\theta.$ Suppose that $\theta'_{kl}=0$ when $k>d$ or when $l>d.$ We have $L_{\infty}(\mathbb{T}^{d'}_{\theta'})=L_{\infty}(\mathbb{T}^d_{\theta})\bar{\otimes}L_{\infty}(\mathbb{T}^{d'-d}).$
\end{ex}
\begin{proof} Let $\{U_k\}_{1\leq k\leq d'}$ be the unitaries in the definition of $\mathbb{T}^{d'}_{\theta}.$ Note that
\begin{enumerate}[{\rm (1)}]
\item elements $\{U_k\}_{1\leq k\leq d}$ generate the algebra $A_{\theta};$
\item elements $\{U_k\}_{d<k\leq d'}$ generate the algebra $A_0;$
\item if $1\leq k\leq d$ and $l>d,$ then $U_k$ commutes with $U_l;$
\end{enumerate}

Define trace preserving $\ast-$isomorphism $\pi:A_{\theta}\otimes A_0\to A_{\theta'}$ by setting
$$\pi(U_k\otimes U_l)=U_kU_l,\quad 1\leq k\leq d,\quad d<l\leq d'.$$
Since $\pi$ preserves the trace, it follows that
$$\langle \pi(x),\pi(y)\rangle=\langle x,y\rangle,\quad x,y\in A_{\theta}\otimes A_0.$$
Thus, $\pi$ extends to a Hilbert space isomorphism $U:H_{\theta}\otimes H_0\to H_{\theta'}.$ 

It is immediate that
$$\lambda_l(\pi(x))=U(\lambda_l\otimes\lambda_l)(x)U^{-1},\quad x\in A_{\theta}\otimes A_0.$$
Hence, the mapping $z\to UzU^{-1}$ delivers a $\ast-$isomorphism from the algebra $L_{\infty}(\mathbb{T}^d_{\theta})\bar{\otimes}L_{\infty}(\mathbb{T}^{d'-d})$ to $L_{\infty}(\mathbb{T}^{d'}_{\theta'}).$
\end{proof}

As usual, $L_p(\mathbb{T}^d_{\theta})$ is the $L_p-$space associated to the von Neumann algebra $L_{\infty}(\mathbb{T}^d_{\theta})$ and the trace $\tau.$

The Hilbert space $H$ is naturally identified with $L_2(\mathbb{T}^d_{\theta}).$ Every element $x\in L_2(\mathbb{T}^d_{\theta})$ admits a unique representation of the form
$$x=\sum_{n\in\mathbb{Z}^d}\hat{x}(n)e_n,\quad \{\hat{x}(n)\}_{n\in\mathbb{Z}^d}\in l_2(\mathbb{Z}^d).$$
This Fourier picture allows us to define Sobolev spaces $W^{k,2}(\mathbb{T}^d_{\theta})$ by setting
$$W^{k,2}(\mathbb{T}^d_{\theta})=\Big\{x\in L_2(\mathbb{T}^d_{\theta}):\ \sum_{n\in\mathbb{Z}^d}|n|_2^{2k}|\hat{x}(n)|^2<\infty\Big\}.$$

For $1\leq k\leq d,$ define self-adjoint operators $D_k:W^{1,2}(\mathbb{T}^d_{\theta})\to L_2(\mathbb{T}^d_{\theta})$ by setting
$$D_k(x)=\sum_{n\in\mathbb{Z}^d}n_k\hat{x}(n)e_n.$$

\begin{fact}\label{basic exponential fact} We have (the second equality holds for all $x\in L_{\infty}(\mathbb{T}^d_{\theta})$)
$$D_i\lambda_r(e_n)=\lambda_r(e_n)D_i+n_i\lambda_r(e_n),\quad \lambda_l(x)\lambda_r(e_n)=\lambda_r(e_n)\lambda_l(x).$$
\end{fact}
\begin{proof} Second equality is obvious. Let's check the first equality. Recall that
$$e_me_n=c_{m,n}e_{m+n}.$$
We have
$$(D_i\lambda_r(e_n))(e_m)=D_i(e_me_n)=c_{m,n}D_i(e_{m+n})=c_{m,n}(m_i+n_i)e_{m+n}=$$
$$=(m_i+n_i)e_me_n=(m_ie_m)\cdot e_n+n_i\cdot e_me_n=(\lambda_r(e_n)D_i+n_i\lambda_r(e_n))(e_m).$$
\end{proof}

Set
$$D^{\alpha}(x)=(\prod_{k=1}^dD_k^{\alpha_k})(x),\quad \alpha=(\alpha_1,\cdots,\alpha_d)\in\mathbb{Z}_+^d.$$
Set
$$C^{\infty}(\mathbb{T}^d_{\theta})=\Big\{x\in L_{\infty}(\mathbb{T}^d_{\theta}):\ D^{\alpha}(x)\in L_{\infty}(\mathbb{T}^d_{\theta})\mbox{ for all }\alpha\Big\},$$
$$W^{k,p}(\mathbb{T}^d_{\theta})=\Big\{x\in L_p(\mathbb{T}^d_{\theta}):\ D^{\alpha}x\in L_p(\mathbb{T}^d_{\theta}),\quad |\alpha|_1\leq k\Big\},\quad p>0,\quad k\in\mathbb{Z}_+.$$

Here, $|\alpha|_1$ is the $l_1-$length of the vector $\alpha\in\mathbb{Z}^d.$

We equip $W^{k,p}(\mathbb{T}^d_{\theta})$ with its natural norm
\begin{equation}\label{sobolev norm def}
\|x\|_{W^{k,p}}=\sum_{|\alpha|_1\leq k}\|D^{\alpha}x\|_p.
\end{equation}

For $p=2,$ the space $W^{k,p}(\mathbb{T}^d_{\theta})$ coincides with earlier defined $W^{k,2}(\mathbb{T}^d_{\theta}).$ Indeed,
$$D^{\alpha}(\sum_{n\in\mathbb{Z}^d}\hat{x}(n)e_n)=\sum_{n\in\mathbb{Z}^d}n^{\alpha}\hat{x}(n)e_n,\quad n^{\alpha}=\prod_{k=1}^dn_k^{\alpha_k}.$$
Hence,
$$\|D^{\alpha}x\|_2=\Big(\sum_{n\in\mathbb{Z}^d}|n^{\alpha}|^2\cdot|\hat{x}(n)|^2\Big)^{\frac12}.$$
Therefore,
$$\sum_{|\alpha|_1\leq k}\|D^{\alpha}x\|_2=\sum_{|\alpha|_1\leq k}\Big(\sum_{n\in\mathbb{Z}^d}|n^{\alpha}|^2\cdot|\hat{x}(n)|^2\Big)^{\frac12}.$$
Obviously,
$$|n^{\alpha}|\leq|n|_2^{|\alpha|_1},\quad n\in\mathbb{Z}^d,\quad \alpha\in\mathbb{Z}_+^d,$$
and, therefore
$$\sum_{|\alpha|_1\leq k}\|D^{\alpha}x\|_2\leq\sum_{|\alpha|_1\leq k}\Big(\sum_{n\in\mathbb{Z}^d}|n|_2^{2k}\cdot|\hat{x}(n)|^2\Big)^{\frac12}=\Big(\sum_{n\in\mathbb{Z}^d}|n|_2^{2k}\cdot|\hat{x}(n)|^2\Big)^{\frac12}\cdot\sum_{|\alpha|_1\leq k}1.$$
On the other hand, we can consider only
$$\alpha=(k,0,\cdots,0),\quad \alpha=(0,k,0,\cdots,0),\cdots.$$
We have
$$\sum_{|\alpha|_1\leq k}\|D^{\alpha}x\|_2\geq\sum_{l=1}^d\Big(\sum_{n\in\mathbb{Z}^d}|n_l|^{2k}\cdot|\hat{x}(n)|^2\Big)^{\frac12}\geq$$
$$\geq\Big(\sum_{l=1}^d\sum_{n\in\mathbb{Z}^d}|n_l|^{2k}\cdot|\hat{x}(n)|^2\Big)^{\frac12}
=\Big(\sum_{n\in\mathbb{Z}^d}|\hat{x}(n)|^2\cdot\sum_{l=1}^d|n_l|^{2k}\Big)^{\frac12}.$$
On the other hand, we have
$$\sum_{l=1}^d|n_l|^{2k}\geq d^{1-k}|n|_2^{2k},\quad n\in\mathbb{Z}^d.$$
Hence,
$$\sum_{|\alpha|_1\leq k}\|D^{\alpha}x\|_2\geq d^{\frac{1-k}{2}}\Big(\sum_{n\in\mathbb{Z}^d}|n|_2^{2k}\cdot|\hat{x}(n)|^2\Big)^{\frac12}.$$
Thus,
$$\sum_{|\alpha|_1\leq k}\|D^{\alpha}x\|_2\approx\Big(\sum_{n\in\mathbb{Z}^d}|n|_2^{2k}\cdot|\hat{x}(n)|^2\Big)^{\frac12}.$$

We also set
$$\Delta=\sum_{k=1}^dD_k^2.$$
Obviously,
$$\Delta:W^{2,2}(\mathbb{T}^d_{\theta})\to L_2(\mathbb{T}^d_{\theta})$$
is self-adjoint (and positive).

\section{Definition of a curved non-commutative torus}

\subsection{Curved non-commutative torus}

Here we define curved non-commutative torus and Laplace-Beltrami operator on it.

For a positive invertible element $\nu\in L_{\infty}(\mathbb{T}^d_{\theta}),$ consider
$$\phi_{\nu}:x\to \tau(x\nu),\quad x\in L_{\infty}(\mathbb{T}^d_{\theta}).$$
Define a new inner product on $L_2(\mathbb{T}^d_{\theta})$ by setting
$$\langle u,v\rangle_{\nu}=\phi_{\nu}(vu^*)=\tau(u^*\nu v).$$

Let ${\rm GL}_d(L_{\infty}(\mathbb{T}^d_{\theta}))$ be the set of invertible matrices with matrix elements from $L_{\infty}(\mathbb{T}^d_{\theta}).$ Let ${\rm GL}_d(C^{\infty}(\mathbb{T}^d_{\theta}))$ be the set of invertible matrices with matrix elements from $C^{\infty}(\mathbb{T}^d_{\theta}).$

Riemannian metric on $\mathbb{T}^d_{\theta}$  (see \cite{Rosenberg_sigma} or \cite{PongeHa}) is a positive element of ${\rm GL}_d(C^{\infty}(\mathbb{T}^d_{\theta}))$ such that the elements $g_{ij}$ and $(g^{-1})_{ij}$ are self-adjoint for all $1\leq i,j\leq d$  (see \cite{PongeHa}).

\subsection{Laplace-Beltrami operator}\label{lbo subsection}

Let $g\in {\rm GL}_d(C^{\infty}(\mathbb{T}^d_{\theta}))$ be a Riemannian metric. In the classical differential geometry, Laplace-Beltrami operator involves the square root of the determinant of $g.$ In the non-commutative case, there is no notion of a determinant for a matrix with non-commuting elements. We propose the following substitution for a "square root of the determinant" of $g.$ Set 
\begin{equation}\label{nu def}
\nu=\pi^{-\frac{d}{2}}\int_{\mathbb{R}^d}e^{-\sum_{i,j=1}^dt_it_j(g^{-1})_{ij}}dt.
\end{equation}
Note that $(g^{-1})_{ij}\in C^{\infty}(\mathbb{T}^d_{\theta})$ (see \cite{PongeHa}). Hence, $\sum_{i,j=1}^dt_it_j(g^{-1})_{ij}\in C^{\infty}(\mathbb{T}^d_{\theta})$ for all $t\in\mathbb{R}^d.$ It follows that $e^{-\sum_{i,j=1}^dt_it_j(g^{-1})_{ij}}\in C^{\infty}(\mathbb{T}^d_{\theta})$ for all $t\in\mathbb{R}^d.$ Moreover, the integrand in \eqref{nu def} is Bochner integrable in every $C^m(\mathbb{T}^d_{\theta}),$ $m\geq0.$ Thus, $\nu\in C^{\infty}(\mathbb{T}^d_{\theta})$ and $\nu^{\frac12}\in C^{\infty}(\mathbb{T}^d_{\theta}).$

This choice of $\nu$ may seem unexpected, however it appears to be very natural. In fact, this is the only choice of $\nu$ which makes the Laplace-Beltrami operator defined below compatible with Connes Integration Formula (see \cite{MPSZ}).

Laplace-Beltrami operator $\Delta_g$ is defined on the Hilbert space $L_2(\mathbb{T}^d_{\theta})$ equipped with the inner product $\langle\cdot,\cdot\rangle_{\nu}$ by the formula  (see \cite{PongeHa} or \cite{Rosenberg_sigma})
$$\Delta_g=\lambda_l(\nu^{-1})\sum_{i,j=1}^dD_i\lambda_l(\nu^{\frac12}(g^{-1})_{ij}\nu^{\frac12})D_j.$$
Laplace-Beltrami operator is self-adjoint and positive on the domain $W^{2,2}(\mathbb{T}^d_{\theta})$  (see Proposition 9.12 in \cite{PongeHa}).

\subsection{Statement of the task}

The task is to find the asymptotic for the function
$$t\to {\rm Tr}(\lambda_l(x)e^{-t\Delta_g}),\quad t\downarrow0.$$
Here, $x\in L_{\infty}(\mathbb{T}^d_{\theta})$ and $g\in {\rm GL}_d(C^{\infty}(\mathbb{T}^d_{\theta}))$ is a Riemannian metric.

First, note that the mapping $U=\lambda_l(\nu^{-\frac12})$ is a unitary operator from $(L_2(\mathbb{T}^d_{\theta}),\langle\cdot,\cdot\rangle)$ to $(L_2(\mathbb{T}^d_{\theta}),\langle\cdot,\cdot\rangle_{\nu})$ (this follows directly from the definition of these inner products). Define a self-adjoint (and positive) operator $A_g$ on the Hilbert space $(L_2(\mathbb{T}^d_{\theta}),\langle\cdot,\cdot\rangle)$ with the domain $W^{2,2}(\mathbb{T}^d_{\theta})$ by setting
$$A_g=U^{-1}\Delta_gU.$$
Equivalently,
\begin{equation}\label{ag expression}
A_g=\lambda_l(\nu^{-\frac12})\sum_{i,j=1}^dD_i\lambda_l(\nu^{\frac12}(g^{-1})_{ij}\nu^{\frac12})D_j\lambda_l(\nu^{-\frac12}).
\end{equation}

\begin{ex}[Conformal deformation of a flat torus]
For example, if $d=2$ and $g=(h\delta_{ij}),$ then
$$A_g=\lambda_l(h^{-\frac12})\Delta\lambda_l(h^{-\frac12})$$
exactly as it should be.
\end{ex}
\begin{proof} Obviously, $(g^{-1})_{ij}=h^{-1}\delta_{ij}.$ Hence,
$$\nu=\frac1{\pi}\int_{\mathbb{R}^2}e^{-|t|^2h^{-1}}dt=h.$$
Hence,
$$\nu^{\frac12}(g^{-1})_{ij}\nu^{\frac12}=\delta_{ij}.$$
This completes the proof.
\end{proof}

The task can be now equivalently restated as follows: to find an asymptotic for the function
$$t\to {\rm Tr}(\lambda_l(\nu^{\frac12}x\nu^{-\frac12})e^{-tA_g}),\quad t\downarrow0.$$
Here, $x\in L_{\infty}(\mathbb{T}^d_{\theta}),$ $g\in {\rm GL}_d(C^{\infty}(\mathbb{T}^d_{\theta}))$ is a Riemannian metric and $\nu$ is defined by \eqref{nu def}.

Indeed, we have
$${\rm Tr}(\lambda_l(x)e^{-t\Delta_g})={\rm Tr}(U^{-1}\lambda_l(x)e^{-t\Delta_g}U)=$$
$$={\rm Tr}(U^{-1}\lambda_l(x)U\cdot U^{-1}e^{-t\Delta_g}U)={\rm Tr}(\lambda_l(\nu^{\frac12}x\nu^{-\frac12})e^{-tA_g}).$$

\section{Definitions and notations}\label{defnot section}


In this short section, we introduce the notations used in the statement and proof of Theorem \ref{main result}, particularly, the functions ${\rm good}_k$ and ${\rm bad}_n.$

\begin{nota}\label{first resolvent nota} For $s\in\mathbb{R}^d,$ set
$$V(s)=\sum_{i=1}^ds_iA_i,$$
where
$$A_i=\sum_{j=1}^d\lambda_l((g^{-1})_{ij}\nu^{\frac12})D_j\lambda_l(\nu^{-\frac12})+\sum_{j=1}^d\lambda_l(\nu^{-\frac12})D_j\lambda_l(\nu^{\frac12}(g^{-1})_{ji}),\quad 1\leq i\leq d.$$
\end{nota}

\begin{nota}\label{second resolvent nota} Let $g=(g_{ij})\in{\rm GL}_d(C^{\infty}(\mathbb{T}^d_{\theta}))$ be a Riemannian metric. For every $s\in\mathbb{R}^d,$ set
$$x(s)=\sum_{i,j=1}^d(g^{-1})_{ij}s_is_j.$$
\end{nota}

\begin{nota}\label{third resolvent nota} For every $z\in\mathbb{C}\backslash\mathbb{R}_-$ and for every $s\in\mathbb{R}^d,$ set $x_0(s,z)=1$ and
$$x_m(s,z)=(V(s)+A_g)((x(s)+z)^{-1}x_{m-1}(s,z)),\quad m\in\mathbb{N}.$$
Here, $A_g$ is defined in \eqref{ag expression}.
\end{nota}

\begin{nota}\label{fourth resolvent nota} Let $\mathscr{A}\subset\mathbb{N}.$ For every $z\in\mathbb{C}\backslash\mathbb{R}_-$ and for every $s\in\mathbb{R}^d,$ set
 $x_0^{\mathscr{A}}(s,z)=1$ and
$$x_m^{\mathscr{A}}(s,z)=(V(s))((x(s)+z)^{-1}x_{m-1}^{\mathscr{A}}(s,z)),\quad 1\leq m\in\mathscr{A},$$
$$x_m^{\mathscr{A}}(s,z)=A_g((x(s)+z)^{-1}x_{m-1}^{\mathscr{A}}(s,z)),\quad 1\leq m\notin\mathscr{A}.$$
\end{nota}

Observe that, for $\mathscr{A}\subset\{1,\cdots,m\},$ we have
\begin{equation}\label{xma homogeneity}
x_m^{\mathscr{A}}(rs,r^2z)=r^{|\mathscr{A}|-2m}x_m^{\mathscr{A}}(s,z).
\end{equation}

\begin{nota}\label{fifth resolvent nota}
For every $z\in\mathbb{C}\backslash\mathbb{R}_-$ and for every $s\in\mathbb{R}^d,$ set
$${\rm good}_k(s,z)=(x(s)+z)^{-1}\sum_{\frac{k}{2}\leq m\leq \min(k,d)}(-1)^m\sum_{\substack{\mathscr{A}\subset\{1,\cdots,m\}\\ |\mathscr{A}|=2m-k}}x_m^{\mathscr{A}}(s,z),$$
$${\rm corr}_k(s,z)=(x(s)+z)^{-1}\sum_{\frac{k}{2}\leq m\leq k}(-1)^m\sum_{\substack{\mathscr{A}\subset\{1,\cdots,m\}\\ |\mathscr{A}|=2m-k}}x_m^{\mathscr{A}}(s,z),$$
$${\rm bad}_n(z)=\Big(\lambda_r(e_n)^*\frac1{A_g+z}\lambda_r(e_n)\Big)(x_{d+1}(n,z)).$$
\end{nota}
Obviously, ${\rm good}_k={\rm corr}_k$ for $k\leq d.$

Key feature of the term ${\rm good}_k$ is its homogeneity
\begin{equation}\label{goodk homogeneity}
{\rm good}_k(rs,r^2z)=r^{-k-2}{\rm good}_k(s,z),
\end{equation}
which follows immediately from \eqref{xma homogeneity}.

\begin{nota}\label{first exponential notation} For every $0\neq s\in\mathbb{R}^d$ set
$${\rm Good}_k(s)=\frac1{2\pi}\int_{\mathbb{R}}{\rm good}_k(s,i\lambda)e^{i\lambda}d\lambda.$$
$${\rm Corr}_k(s)=\frac1{2\pi}\int_{\mathbb{R}}{\rm corr}_k(s,i\lambda)e^{i\lambda}d\lambda.$$
For $k>0,$ integrals are well defined; for $k=0,$ integrals should be understood in the sense of principal value.
\end{nota}
Obviously, ${\rm Good}_k={\rm Corr}_k$ for $k\leq d.$

\begin{nota}\label{ik notation} For every $k\in\mathbb{Z}_+,$ we set
$$I_k=\int_{\mathbb{R}^d}{\rm Corr}_k(s)ds.$$
\end{nota}

\section{Strategy}

In the subsequent lemma, weak convergence is asserted, not assumed.

\begin{lem}\label{first strategy lemma} Let $g\in {\rm GL}_d(C^{\infty}(\mathbb{T}^d_{\theta}))$ be a Riemannian metric and let $A_g$ be the operator defined by \eqref{ag expression}. For every $x\in L_2(\mathbb{T}^d_{\theta}),$ we have
$${\rm Tr}(\lambda_l(x)e^{-tA_g})=\tau(x\cdot F(t)).$$
Here $F(t)\in L_2(\mathbb{T}^d_{\theta})$ is given by the series (converging weakly in $L_2(\mathbb{T}^d_{\theta})$)
$$F(t)=\sum_{n\in\mathbb{Z}^d}(\lambda_r(e_n)^*e^{-tA_g}\lambda_r(e_n))(1).$$
\end{lem}
\begin{proof} It follows from \cite{MPSZ} that
$$\lambda_l(x)e^{-tA_g}\in\mathcal{L}_1$$
for every $x\in L_2(\mathbb{T}^d_{\theta}).$

For every $T\in\mathcal{L}_1,$ we have
$${\rm Tr}(T)=\sum_{n\in\mathbb{Z}^d}\langle e_n,Te_n\rangle.$$
Therefore,
$${\rm Tr}(\lambda_l(x)e^{-tA_g})=\sum_{n\in\mathbb{Z}^d}\langle e_n,(\lambda_l(x)e^{-tA_g})(e_n)\rangle.$$
Since $e_n=\lambda_r(e_n)1$ and since $\lambda_r(e_n)$ commutes with $\lambda_l(x),$ it follows that
$$\langle e_n,(\lambda_l(x)e^{-tA_g})(e_n)\rangle=\langle (\lambda_r(e_n))(1),(\lambda_l(x)e^{-tA_g}\lambda_r(e_n))(1)\rangle=$$
$$=\langle 1, (\lambda_r(e_n)^{\ast}\lambda_l(x)e^{-tA_g}\lambda_r(e_n))(1)\rangle=\langle 1, (\lambda_l(x)\lambda_r(e_n)^{\ast}e^{-tA_g}\lambda_r(e_n))(1)\rangle.$$
Combining these equalities, we obtain
$${\rm Tr}(\lambda_l(x)e^{-tA_g})=\sum_{n\in\mathbb{Z}^d}\langle 1,(\lambda_l(x)\lambda_r(e_n)^{\ast}e^{-tA_g}\lambda_r(e_n))(1)\rangle=$$
$$=\Big\langle 1,\lambda_l(x)\Big(\sum_{n\in\mathbb{Z}^d}\lambda_r(e_n)^{\ast}e^{-tA_g}\lambda_r(e_n)1\Big)\Big\rangle=\langle 1,\lambda_l(x)(F(t))\rangle=\tau(x\cdot F(t)).$$
\end{proof}

In Section \ref{resolvent splitting section}, we prove the following result.

\begin{thm}\label{resolvent splitting theorem} For every $0\neq n\in\mathbb{Z}^d$ and for every $0\neq z\in\mathbb{C}\backslash\mathbb{R}_-,$ we have
\begin{equation}\label{resolvent splitting eq}
\Big(\lambda_r(e_n)^*\frac1{A_g+z}\lambda_r(e_n)\Big)(1)=\sum_{k=0}^{2d}{\rm good}_k(n,z)+(-1)^{d+1}{\rm bad}_n(z),
\end{equation}
where
\begin{enumerate}[{\rm (i)}]
\item\label{rspa} functions ${\rm good}_k$ and ${\rm bad}_n$ are explicitly defined in Notation \ref{fifth resolvent nota}.
\item\label{rspb} functions ${\rm good}_k(s,\cdot),$ $k\geq0,$ are analytic on $\mathbb{C}\backslash\mathbb{R}_-.$
\item\label{rspc} for every $z$ with $\Re(z)\leq0,$ we have
$$\|{\rm bad}_n(z)\|_2=O\Big(|z|^{-1}(|n|^2+|z|)^{-\frac{d+1}{2}}\Big),\quad n\in\mathbb{Z}^d.$$
\end{enumerate}
\end{thm}

Functions ${\rm good}_k$ are called good because they have a very concrete representation. Bad terms ${\rm bad}_n$ are, in a certain sense, negligible (see the explanation after Theorem \ref{exponent splitting theorem}).

In Section \ref{exponent splitting section}, we prove the following result.

\begin{thm}\label{exponent splitting theorem} We have
\begin{equation}\label{exponent splitting eq}
(\lambda_r(e_n)^*e^{-tA_g}\lambda_r(e_n))(1)=\sum_{k=0}^{2d}t^{\frac{k}{2}}{\rm Good}_k(nt^{\frac12})+(-1)^{d+1}{\rm Bad}_n(t),
\end{equation}
where
\begin{enumerate}[{\rm (i)}]
\item functions ${\rm Good}_k$ are explicitly defined in Notation \ref{first exponential notation}.
\item we have
$$\|{\rm Bad}_n(t)\|_2=O\Big(\frac{\log(|n|)}{|n|^{d+1}}\Big)$$
uniformly in $t.$
\end{enumerate}
\end{thm}

Terms ${\rm Good}_k$ are called good because they have very concrete representation. Bad terms ${\rm Bad}_n(t)$ are negligible in the following sense:
$$\|\sum_{n\in\mathbb{Z}^d}{\rm Bad}_n(t)\|_2=O(1)$$
uniformly in $t.$

Alain Connes suggested to us the method based on the Poisson summation formula which allows to replace sums with integrals (see Proposition 2.27 in \cite{EcksteinIochum}).

\begin{lem}[Connes "little lemma"]\label{cll} If $f$ is a Schwartz function, then
$$\sum_{n\in\mathbb{Z}^d}f(nt)=t^{-d}\int_{\mathbb{R}^d}f(s)ds+O(t^{\infty}),\quad t>0.$$
\end{lem}
\begin{proof} By Poisson summation formula, we have
$$\sum_{n\in\mathbb{Z}^d}f(nt)=t^{-d}\sum_{n\in\mathbb{Z}^d}(\mathcal{F}f)(nt^{-1}).$$
Here,
$$(\mathcal{F}f)(s)=\int_{\mathbb{R}^d}f(u)e^{-2\pi i\langle u,s\rangle}du.$$
Note that $f$ is a Schwartz function and so is $\mathcal{F}f.$ For every $m,$ we have $(\mathcal{F}f)(s)=O(|s|^{-m}).$ Thus,
$$\sum_{0\neq n\in\mathbb{Z}^d}(\mathcal{F}f)(nt^{-1})=\sum_{0\neq n\in\mathbb{Z}^d}O(\frac{t^m}{|n|^m})=O(t^m),\quad m>d.$$
Hence,
$$\sum_{n\in\mathbb{Z}^d}f(nt)=t^{-d}\int_{\mathbb{R}^d}f(s)ds+O(t^{m-d}),\quad t>0.$$
Since $m$ is arbitrarily large, the assertion follows.
\end{proof}

For vector-valued functions, the notion analogous to that of Schwartz function does not exist (see though a substitute in Appendix C in \cite{PongeHa}). However, the following adjustment of Connes "little lemma" is possible (and proved in Section \ref{poisson section}).

\begin{thm}\label{sobolev connes lemma} For every $f\in W^{p,1}(\mathbb{R}^d,X),$ $p>d,$ we have
$$\Big\|\sum_{n\in\mathbb{Z}^d}f(tn)-t^{-d}\int_{\mathbb{R}^d}f(u)du\Big\|_X=O(t^{p-d}),\quad t>0.$$
\end{thm}

In Section \ref{verification section}, we verify the conditions of Theorem \ref{sobolev connes lemma} for $f={\rm Good}_k$ and infer the following intermediate result.

\begin{thm}\label{finite expansion thm} Let $d\geq 2.$ For every $x\in L_2(\mathbb{T}^d_{\theta}),$ we have
$$\Big|{\rm Tr}(\lambda_l(x)e^{-tA_g})-\sum_{\substack{0\leq k<d\\ k=0{\rm mod}2}}t^{\frac{k-d}{2}}\cdot\tau(xI_k)\Big|=O(1).$$
\end{thm}

\begin{rem} It is tempting to say that, taking more terms in Theorem \ref{resolvent splitting theorem}, we should obtain an asymptotic for ${\rm Tr}(\lambda_l(x)e^{-tA_g})$ modulo $O(t^N)$ for large $N$ (rather than asymptotic modulo $O(t^0)$). However, proving this does not seem straightforward. In our proof of Theorem \ref{main result}, we use tensoring trick instead. 
\end{rem}

\section{Splitting theorem for resolvent in arbitrary dimension}\label{resolvent splitting section}

\begin{lem}\label{ag conjugation lemma} For every $n\in\mathbb{Z}^d,$ we have
$$\lambda_r(e_n)^*A_g\lambda_r(e_n)=\lambda_l(x(n))+A_g+V(n).$$
\end{lem}
\begin{proof} It follows from Fact \ref{basic exponential fact} that
$$\lambda_l(\nu^{-\frac12})D_i\lambda_l(\nu^{\frac12}(g^{-1})_{ij}\nu^{\frac12})D_j\lambda_l(\nu^{-\frac12})\cdot\lambda_r(e_n)=$$
$$=\lambda_r(e_n)\cdot \lambda_l(\nu^{-\frac12})(D_i+n_i)\lambda_l(\nu^{\frac12}(g^{-1})_{ij}\nu^{\frac12})(D_j+n_j)\lambda_l(\nu^{-\frac12}).$$
Hence,
$$\lambda_r(e_n)^*A_g\lambda_r(e_n)=A_g+\lambda_l(\sum_{i,j=1}^dn_in_j(g^{-1})_{ij})+$$
$$+\sum_{i,j=1}^d\lambda_l(\nu^{-\frac12})n_i\lambda_l(\nu^{\frac12}(g^{-1})_{ij}\nu^{\frac12})D_j\lambda_l(\nu^{-\frac12})+$$
$$+\sum_{i,j=1}^d\lambda_l(\nu^{-\frac12})D_i\lambda_l(\nu^{\frac12}(g^{-1})_{ij}\nu^{\frac12})n_j\lambda_l(\nu^{-\frac12}).$$
\end{proof}

\begin{lem}\label{splitting algebraic lemma} For every $z\in\mathbb{C}\backslash\mathbb{R}_-,$ we have
$$\Big(\lambda_r(e_n)^*\frac1{A_g+z}\lambda_r(e_n)\Big)(1)=\sum_{k=0}^{2d}{\rm good}_k(n,z)+(-1)^{d+1}{\rm bad}_n(z).$$
\end{lem}
\begin{proof} Iterating the resolvent identity, we obtain
$$\frac1{A+z}=\sum_{m=0}^d(-1)^m\frac1{B+z}\cdot \Big((A-B)\frac1{B+z}\Big)^m+\frac{(-1)^{d+1}}{A+z}\cdot \Big((A-B)\frac1{B+z}\Big)^{d+1}.$$

Now, we set $A=\lambda_r(e_n)^*A_g\lambda_r(e_n)$ and $B=\lambda_l(x(n))$ and apply both sides to the vector $1.$ By Lemma \ref{ag conjugation lemma}, we have
$$A-B=V(n)+A_g.$$
Using the equality
$$\Big((A-B)\frac1{B+z}\Big)^m(1)=\Big((V(n)+A_g)\lambda_l((x(n)+z)^{-1})\Big)^m(1)=x_m(n,z),$$
we obtain
\begin{equation}\label{splitting algebraic eq1}
\Big(\frac1{A+z}\Big)(1)=\sum_{m=0}^d(-1)^m(x(n)+z)^{-1}x_m(n,z)+(-1)^{d+1}{\rm bad}_n(z).
\end{equation}

Finally, we have
$$x_m(n,z)=\sum_{\mathscr{A}\subset\{1,\cdots,m\}}x_m^{\mathscr{A}}(n,z).$$
Thus,
$$\sum_{m=0}^d(-1)^m(x(n)+z)^{-1}x_m(n,z)=\sum_{m=0}^d(-1)^m(x(n)+z)^{-1}\sum_{\mathscr{A}\subset\{1,\cdots,m\}}x_m^{\mathscr{A}}(n,z).$$
Obviously,
$$\sum_{\mathscr{A}\subset\{1,\cdots,m\}}x_m^{\mathscr{A}}(n,z)=\sum_{k=m}^{2m}\sum_{\substack{\mathscr{A}\subset\{1,\cdots,m\}\\ |\mathscr{A}|=2m-k}}x_m^{\mathscr{A}}(n,z).$$
Therefore,
$$\sum_{m=0}^d(-1)^m(x(n)+z)^{-1}x_m(n,z)=$$
$$=\sum_{m=0}^d\sum_{k=m}^{2m}(-1)^m(x(n)+z)^{-1}\sum_{\substack{\mathscr{A}\subset\{1,\cdots,m\}\\ |\mathscr{A}|=2m-k}}x_m^{\mathscr{A}}(n,z)$$
$$=\sum_{k=0}^{2d}\sum_{\frac{k}{2}\leq m\leq \min(k,d)}(-1)^m(x(n)+z)^{-1}\sum_{\substack{\mathscr{A}\subset\{1,\cdots,m\}\\ |\mathscr{A}|=2m-k}}x_m^{\mathscr{A}}(n,z)=\sum_{k=0}^{2d}{\rm good}_k(n,z).$$
Combining the last equality with \eqref{splitting algebraic eq1}, we complete the proof.
\end{proof}

In the following lemmas, $c_k(g)$ are some constants (they may differ in different lemmas) which depend only on $k$ and the metric $g.$ Their precise values are irrelevant.

\begin{lem}\label{first sobolev lemma} For every $k\geq0,$ we have
$$\|(x(s)+z)^{-1}\|_{W^{k,\infty}}\leq\frac{c_k(g)}{|s|^2+|z|},\quad\Re(z)\leq 0,\quad s\in\mathbb{R}^d.$$
\end{lem}
\begin{proof} We prove the assertion by induction on $k.$ For $k=0,$ we have $W^{0,\infty}=L_{\infty}$ and the assertion is obvious. Suppose, it is true for $k$ and let us prove it for $k+1.$

By definition \eqref{sobolev norm def}, we have
$$\|x\|_{W^{k+1,p}}\leq\|x\|_{W^{k,p}}+\sum_{j=1}^d\|D_jx\|_{W^{k,p}}.$$
Thus,
$$\|(x(s)+z)^{-1}\|_{W^{k+1,\infty}}\leq \|(x(s)+z)^{-1}\|_{W^{k,\infty}}+\sum_{j=1}^d\|D_j((x(s)+z)^{-1})\|_{W^{k,\infty}}.$$

Clearly,
$$D_j((x(s)+z)^{-1})=-(x(s)+z)^{-1}D_j(x(s))(x(s)+z)^{-1}.$$
Using the inequality
$$\|xy\|_{W^{k,\infty}}\leq 2^k\|x\|_{W^{k,\infty}}\|y\|_{W^{k,\infty}},$$
we arrive at
$$\|D_j((x(s)+z)^{-1})\|_{W^{k,\infty}}\leq 2^{2k}\Big\|(x(s)+z)^{-1}\Big\|_{W^{k,\infty}}^2\cdot \Big\|D_j(x(s))\Big\|_{W^{k,\infty}}.$$
Obviously,
$$\Big\|D_j(x(s))\Big\|_{W^{k,\infty}}\leq\|x(s)\|_{W^{k+1,\infty}}\leq |s|^2\cdot\sum_{i,j=1}^d\|(g^{-1})_{ij}\|_{W^{k+1,\infty}}.$$
Using the inductive assumption, we obtain
$$\|D_j((x(s)+z)^{-1})\|_{W^{k,\infty}}\leq\frac{2^{2k}c_k^2(g)|s|^2}{(|s|^2+|z|)^2}\cdot \sum_{i,j=1}^d\|(g^{-1})_{ij}\|_{W^{k+1,\infty}}=\frac{c_k'(g)|s|^2}{(|s|^2+|z|)^2}.$$
Hence,
$$\|(x(s)+z)^{-1}\|_{W^{k+1,\infty}}\leq\frac{c_k(g)}{|s|^2+|z|}+\frac{dc_k'(g)|s|^2}{(|s|^2+|z|)^2}\leq \frac{c_k(g)+dc_k'(g)}{|s|^2+|z|}.$$
\end{proof}

\begin{lem}\label{second sobolev lemma} For every $(m,k)\geq0,$ we have
$$\|x_m^{\mathscr{A}}(s,z)\|_{W^{k,2}}\leq\frac{c_k(g)}{(|s|^2+|z|)^{\frac12}}\|x_{m-1}^{\mathscr{A}}(s,z)\|_{W^{k+1,2}},\quad m\in\mathscr{A},$$
$$\|x_m^{\mathscr{A}}(s,z)\|_{W^{k,2}}\leq\frac{c_k(g)}{|s|^2+|z|}\|x_{m-1}^{\mathscr{A}}(s,z)\|_{W^{k+2,2}},\quad m\notin\mathscr{A}.$$
\end{lem}
\begin{proof} Consider the case $m\in\mathscr{A}.$ By definition, we have
$$\|x_m^{\mathscr{A}}(s,z)\|_{W^{k,2}}=\|(V(s))((x(s)+z)^{-1}x_{m-1}^{\mathscr{A}}(s,z))\|_{W^{k,2}}.$$
By triangle inequality, we have
$$\|x_m^{\mathscr{A}}(s,z)\|_{W^{k,2}}\leq\sum_{i=1}^d|s_i|\cdot\|A_i((x(s)+z)^{-1}x_{m-1}^{\mathscr{A}}(s,z))\|_{W^{k,2}}.$$
Using obvious inequality
$$\|A_ix\|_{W^{k,2}}\leq c_k'(g)\cdot \|x\|_{W^{k+1,2}},\quad 1\leq i\leq d,$$
we obtain
$$\|x_m^{\mathscr{A}}(s,z)\|_{W^{k,2}}\leq dc_k'(g)|s|\cdot \|(x(s)+z)^{-1}x_{m-1}^{\mathscr{A}}(s,z)\|_{W^{k+1,2}}.$$

Using the inequality
$$\|xy\|_{W^{k+1,2}}\leq 2^{k+1}\|x\|_{W^{k+1,\infty}}\|y\|_{W^{k+1,2}},$$
we arrive at
$$\|x_m^{\mathscr{A}}(s,z)\|_{W^{k,2}}\leq 2^{k+1}dc_k'(g)|s|\cdot\|(x(s)+z)^{-1}\|_{W^{k+1,\infty}}\|x_{m-1}^{\mathscr{A}}(s,z)\|_{W^{k+1,2}}.$$
The assertion for the case $m\in\mathscr{A}$ follows now from Lemma \ref{first sobolev lemma}.

Consider the case $m\notin\mathscr{A}.$ By definition, we have
$$\|x_m^{\mathscr{A}}(s,z)\|_{W^{k,2}}=\|A_g((x(s)+z)^{-1}x_{m-1}^{\mathscr{A}}(s,z))\|_{W^{k,2}}.$$
Using obvious inequality
$$\|A_gx\|_{W^{k,2}}\leq c_k'(g)\|x\|_{W^{k+2,2}},$$
we obtain
$$\|x_m^{\mathscr{A}}(s,z)\|_{W^{k,2}}\leq c_k'(g)\|(x(s)+z)^{-1}x_{m-1}^{\mathscr{A}}(s,z)\|_{W^{k+2,2}}.$$
Using the inequality
$$\|xy\|_{W^{k+2,2}}\leq 2^{k+2}\|x\|_{W^{k+2,\infty}}\|y\|_{W^{k+2,2}},$$
we arrive at
$$\|x_m^{\mathscr{A}}(s,z)\|_{W^{k,2}}\leq 2^{k+2}c_k'(g)\|(x(s)+z)^{-1}\|_{W^{k+2,\infty}}\|x_{m-1}^{\mathscr{A}}(s,z)\|_{W^{k+2,2}}.$$
The assertion for the case $m\notin\mathscr{A}$ follows now from Lemma \ref{first sobolev lemma}.
\end{proof}

\begin{lem}\label{third sobolev lemma} For every $(m,k)\geq0$ and for every $\mathscr{A}\subset\{1,\cdots,m\},$ we have
$$\|x_m^{\mathscr{A}}(s,z)\|_{W^{k,2}}\leq\frac{c_{m,k}(g)}{(|s|^2+|z|)^{m-\frac12|\mathscr{A}|}}.$$
\end{lem}
\begin{proof} The assertion follows by induction on $m.$ For $m=0,$ we have that $\mathscr{A}=\varnothing$ and, hence, $|\mathscr{A}|=0.$ It is immediate that
$$\|x_0^{\varnothing}(s,z)\|_{W^{k,2}}=\|1\|_{W^{k,2}}=1.$$
This establishes base of induction. 

We now establish the step of induction. Suppose the assertion is true for $m-1,$ for every subset of $\{1,\cdots,m-1\}$ and for {\it every} $k.$ Let $\mathscr{B}=\mathscr{A}\backslash\{m\}\subset\{1,\cdots,m-1\}.$ If $m\in\mathscr{A},$ then Lemma \ref{second sobolev lemma} asserts that
 $$\|x_m^{\mathscr{A}}(s,z)\|_{W^{k,2}}\leq\frac{c_k(g)}{(|s|^2+|z|)^{\frac12}}\|x_{m-1}^{\mathscr{A}}(s,z)\|_{W^{k+1,2}}=\frac{c_k(g)}{(|s|^2+|z|)^{\frac12}}\|x_{m-1}^{\mathscr{B}}(s,z)\|_{W^{k+1,2}}.$$
Applying inductive assumption for the set $\mathscr{B},$ we obtain
 $$\|x_m^{\mathscr{A}}(s,z)\|_{W^{k,2}}\leq\frac{c_k(g)}{(|s|^2+|z|)^{\frac12}}\cdot\frac{c_{m-1,k+1}(g)}{(|s|^2+|z|)^{(m-1)-\frac12|\mathscr{B}|}}=\frac{c_k(g)c_{m-1,k+1}(g)}{(|s|^2+|z|)^{m-\frac12|\mathscr{A}|}}.$$
If $m\notin\mathscr{A},$ then Lemma \ref{second sobolev lemma} asserts that
 $$\|x_m^{\mathscr{A}}(s,z)\|_{W^{k,2}}\leq\frac{c_k(g)}{|s|^2+|z|}\|x_{m-1}^{\mathscr{A}}(s,z)\|_{W^{k+2,2}}=\frac{c_k(g)}{|s|^2+|z|}\|x_{m-1}^{\mathscr{B}}(s,z)\|_{W^{k+2,2}}.$$
Applying inductive assumption for the set $\mathscr{B},$ we obtain
 $$\|x_m^{\mathscr{A}}(s,z)\|_{W^{k,2}}\leq\frac{c_k(g)}{|s|^2+|z|}\cdot\frac{c_{m-1,k+2}(g)}{(|s|^2+|z|)^{(m-1)-\frac12|\mathscr{B}|}}=\frac{c_k(g)c_{m-1,k+2}(g)}{(|s|^2+|z|)^{m-\frac12|\mathscr{A}|}}.$$
This establishes step of induction.
\end{proof}

\begin{proof}[Proof of Theorem \ref{resolvent splitting theorem}] Firstly, the equality \eqref{resolvent splitting eq} is established in Lemma \ref{splitting algebraic lemma}. The assertion of Theorem \ref{resolvent splitting theorem} \eqref{rspa} does not require any proof. The assertion of Theorem \ref{resolvent splitting theorem} \eqref{rspb} is immediate from the definition of the term ${\rm good}_k$ (see Notation \ref{fifth resolvent nota}).

It remains to show the assertion of Theorem \ref{resolvent splitting theorem} \eqref{rspc}. By definition (see Notation \ref{fifth resolvent nota}), we have
$${\rm bad}_n(z)=\Big(\lambda_r(e_n)^*\frac1{A_g+z}\lambda_r(e_n)\Big)(x_{d+1}(n,z)).$$
Therefore,
\begin{equation}\label{bad est eq1}
\|{\rm bad}_n(z)\|_2\leq\Big\|\lambda_r(e_n)^*\frac1{A_g+z}\lambda_r(e_n)\Big\|_{L_2\to L_2}\cdot\|x_{d+1}(n,z)\|_2.
\end{equation}

Since $A_g\geq0,$ it follows that
\begin{equation}\label{bad est eq2}
\Big\|\lambda_r(e_n)^*\frac1{A_g+z}\lambda_r(e_n)\Big\|_{L_2\to L_2}=\Big\|\frac1{A_g+z}\Big\|_{L_2\to L_2}\leq |z|^{-1},\quad \Re(z)\geq0.
\end{equation}

On the other hand, it follows from Lemma \ref{third sobolev lemma} (with $k=0$) that
$$\|x_{d+1}(n,z)\|_2\leq\sum_{\mathscr{A}\subset\{1,\cdots,d+1\}}\|x_{d+1}^{\mathscr{A}}(n,z)\|_2\leq c(g)\sum_{\mathscr{A}\subset\{1,\cdots,d+1\}}(|n|^2+|z|)^{\frac12|\mathscr{A}|-d-1}.$$
Since $|\mathscr{A}|\leq d+1$ and $n\neq0,$ it follows that
$$(|n|^2+|z|)^{\frac12|\mathscr{A}|-d-1}\leq\frac{1}{(|n|^2+|z|)^{\frac{d+1}{2}}}.$$
Thus,
\begin{equation}\label{bad est eq3}
\|x_{d+1}(n,z)\|_2\leq \frac{2^{d+1}c(g)}{(|n|^2+|z|)^{\frac{d+1}{2}}}.
\end{equation}
Combining \eqref{bad est eq1}, \eqref{bad est eq2} and \eqref{bad est eq3}, we arrive at
$$\|{\rm bad}_n(z)\|_2\leq \frac{2^{d+1}c(g)}{|z|\cdot(|n|^2+|z|)^{\frac{d+1}{2}}}.$$
This completes the proof of Theorem \ref{resolvent splitting theorem} \eqref{rspc}.
\end{proof}

\section{Splitting theorem for exponential in arbitrary dimension}\label{exponent splitting section}

In this section, $\Gamma$ denotes the contour passing from $-i\infty$ to $i\infty$ as follows: along the line $\{\Re(z)=0\}$ from $-i\infty$ to $-i,$ then along the circle $\{|z|=1\}$ in the counter-clock-wise direction from $-i$ to $i,$ then along the line $\{\Re(z)=0\}$ from $i$ to $i\infty.$ This contour is introduced with a single purpose: to avoid the origin in the integration. However, in the statement of Theorem \ref{exponent splitting theorem}, we use Notation \ref{first exponential notation}, where the integration is taken over the line $\{\Re z=0\}.$ This allows us to employ homogeneity of the function ${\rm good}_k$ (as in \eqref{goodk homogeneity}) and to write the respective integral as ${\rm Good}_k.$

\begin{lem}\label{2 contour lemma} For $0\neq s\in\mathbb{R}^d$ and for every $t>0,$ we have
$$t^{\frac{k}{2}}{\rm Good}_k(st^{\frac12})=\frac1{2\pi i}\int_{\Gamma}{\rm good}_k(s,z)e^{tz}dz.$$
\end{lem}
\begin{proof} By definition of ${\rm Good}_k$ (see Notation \ref{first exponential notation}), we have
$$t^{\frac{k}{2}}{\rm Good}_k(st^{\frac12})=\frac{t^{\frac{k}{2}}}{2\pi}\int_{\mathbb{R}}{\rm good}_k(t^{\frac12}s,i\lambda)e^{i\lambda}d\lambda.$$
By the homogeneity of ${\rm good}_k,$ we have
$$t^{\frac{k}{2}}{\rm Good}_k(st^{\frac12})=\frac{t^{-1}}{2\pi}\int_{\mathbb{R}}{\rm good}_k(s,it^{-1}\lambda)e^{i\lambda}d\lambda.$$
Changing the variable $\lambda=t\mu,$ we obtain
$$t^{\frac{k}{2}}{\rm Good}_k(st^{\frac12})=\frac{1}{2\pi}\int_{\mathbb{R}}{\rm good}_k(s,i\mu)e^{it\mu}d\mu=\frac1{2\pi i}\int_{i\mathbb{R}}{\rm good}_k(s,z)e^{tz}dz.$$
Thus,
$$t^{\frac{k}{2}}{\rm Good}_k(st^{\frac12})-\frac1{2\pi i}\int_{\Gamma}{\rm good}_k(s,z)e^{tz}dz=\frac1{2\pi i}\int_{\Gamma_1}{\rm good}_k(s,z)e^{tz}dz,$$
where the closed contour $\Gamma_1$ goes from $z=-i$ to $z=i$ along the line $\{\Re(z)=0\},$ then from $z=i$ to $z=-i$ along the circle $\{|z|=1\}$ clockwise.

Note that $x(s)\geq c(g)|s|^2$ for every $s\in\mathbb{R}^d.$ Hence, ${\rm good}_k(s,\cdot)$ is analytic in $\mathbb{C}\backslash(-\infty,-c(g)|s|^2].$ Since $\Gamma_1$ is a closed contour lying inside $\mathbb{C}\backslash(-\infty,-c(g)|s|^2],$ it follows from Cauchy theorem that
$$\int_{\Gamma_1}{\rm good}_k(s,z)e^{tz}dz=0.$$
This completes the proof.
\end{proof}

\begin{proof}[Proof of Theorem \ref{exponent splitting theorem}] For every $y>0$ and $t\geq0,$ we have
$$e^{-ty}=\frac1{2\pi i}{\rm p.v.}\int_{i\mathbb{R}}\frac{e^{tz}dz}{y+z}.$$
Here, principal value is needed because the integral is not absolutely convergent at infinity. Hence, for every $y,t\geq0$ we have
$$e^{-ty}=\frac1{2\pi i}{\rm p.v.}\int_{\Gamma}\frac{e^{tz}dz}{y+z}.$$
By the functional calculus, we have
$$e^{-tA}=\frac1{2\pi i}{\rm p.v.}\int_{\Gamma}\frac{e^{tz}dz}{A+z}.$$

Therefore,
$$(\lambda_r(e_n)^*e^{-tA_g}\lambda_r(e_n))(1)=\frac1{2\pi i}{\rm p.v.}\int_{\Gamma}(\lambda_r(e_n)^*\frac1{A_g+z}\lambda_r(e_n))(1)e^{tz}dz.$$

By Lemma \ref{2 contour lemma}, we have
$$t^{\frac{k}{2}}{\rm Good}_k(st^{\frac12})=\frac1{2\pi i}\int_{\Gamma}{\rm good}_k(s,z)e^{tz}dz.$$
Setting
$${\rm Bad}_n(t)=\frac1{2\pi i}\int_{\Gamma}{\rm bad}_n(z)e^{tz}dz,$$
we infer from Theorem \ref{resolvent splitting theorem} that
$$(\lambda_r(e_n)^*e^{-tA_g}\lambda_r(e_n))(1)=\sum_{k=0}^{2d}t^{\frac{k}{2}}{\rm Good}_k(nt^{\frac12})+(-1)^{d+1}{\rm Bad}_n(t).$$

Obviously,
$$\|{\rm Bad}_n(t)\|_2\leq\frac1{2\pi}\int_{\Gamma}\|{\rm bad}_n(z)\|_2\cdot |e^{tz}|\cdot |dz|.$$
By Theorem \ref{resolvent splitting theorem}, we have
$$\|{\rm bad}_n(z)\|_2\leq\frac{c(g)}{|z|\cdot (|n|^2+|z|)^{\frac{d+1}{2}}}$$
for {\it some} constant $c(g),$ which only depends on $g$ and not on $n.$ On $\Gamma,$ we have $|e^{tz}|\leq e^t\leq e$ as $t\in(0,1).$ Therefore,
$$\|{\rm Bad}_n(t)\|_2\leq c(g)\cdot\int_{\Gamma}\frac{|dz|}{|z|\cdot (|n|^2+|z|)^{\frac{d+1}{2}}}.$$

Clearly,
$$\int_{\Gamma}\frac{|dz|}{|z|\cdot (|n|^2+|z|)^{\frac{d+1}{2}}}=\frac{\pi}{(|n|^2+1)^{\frac{d+1}{2}}}+2\int_1^{\infty}\frac{d\lambda}{\lambda\cdot (|n|^2+\lambda)^{\frac{d+1}{2}}}.$$
Furthermore,
$$\int_1^{|n|^2}\frac{d\lambda}{\lambda\cdot (|n|^2+\lambda)^{\frac{d+1}{2}}}\leq \int_1^{|n|^2}\frac{d\lambda}{\lambda}\cdot |n|^{-d-1}=\frac{2\log(|n|)}{|n|^{d+1}}$$
and
$$\int_{|n|^2}^{\infty}\frac{d\lambda}{\lambda\cdot (|n|^2+\lambda)^{\frac{d+1}{2}}}\leq \int_{|n|^2}^{\infty}\frac{d\lambda}{\lambda^{\frac{d+3}{2}}}=\frac{2}{d+1}|n|^{-d-1}.$$
Combining these inequalities, we obtain
$$\|{\rm Bad}_n(t)\|_2=O\Big(\frac{\log(|n|)}{|n|^{d+1}}\Big),$$
as desired.
\end{proof}

\section{Poisson summation formula for vector-valued functions}\label{poisson section}

\begin{thm}\label{poisson thm} Let $X$ be a Banach space. Poisson summation formula holds for every $f\in W^{p,1}(\mathbb{R}^d,X),$ $p>d,$ i.e.
$$\sum_{n\in\mathbb{Z}^d}f(n)=\sum_{n\in\mathbb{Z}^d}(\mathcal{F}f)(n).$$ 
\end{thm}

In what follows,
$$(\mathcal{F}f)(s)=\int_{\mathbb{R}^d}f(u)e^{-2\pi i\langle u,s\rangle}du,\quad s\in\mathbb{R}^d.$$
$$(T_sf)(u)=f(u-s),\quad u,s\in\mathbb{R}^d.$$

In what follows, $BWC(\mathbb{R}^d,X)$ denotes the space of bounded weak$^{\ast}$ continuous $X-$valued functions on $\mathbb{R}^d.$

\begin{lem}\label{poisson first lemma} We have $W^{d,1}(\mathbb{R}^d,X)\subset BWC(\mathbb{R}^d,X).$ Moreover, we have
\begin{equation}\label{uniform norm ineq}
\|f\|_{L_{\infty}(\mathbb{R}^d,X)}\leq\|f\|_{W^{d,1}(\mathbb{R}^d,X)},\quad f\in W^{d,1}(\mathbb{R}^d).
\end{equation}
\end{lem}
\begin{proof} {\bf Step 1:} Let us prove the inequality \eqref{uniform norm ineq} for $X=\mathbb{C}$ and for every Schwartz function $f.$

For $s\in\mathbb{R}^d,$ let
$$K_s=\{u\in\mathbb{R}^d:\ u\leq s\}.$$
We have
$$f(s)=\int_{K_s}(\partial_0\cdots\partial_{d-1}f)(u)du.$$
Thus,
$$|f(s)|\leq\|\partial_0\cdots\partial_{d-1}f\|_1\leq\|f\|_{W^{d,1}(\mathbb{R}^d)}.$$
Taking supremum over $s\in\mathbb{R}^d,$ we complete the proof of Step 1.

{\bf Step 2:} Let us prove the assertion for $X=\mathbb{C}.$

Now, recall that Schwartz functions are dense in $W^{d,1}(\mathbb{R}^d).$ For a given $f\in W^{d,1}(\mathbb{R}^d),$ choose a sequence $\{f_n\}_{n\geq0}$ of Schwartz functions such that $f_n\to f$ in $W^{d,1}(\mathbb{R}^d)$ (and, therefore, in distributional sense). We have
$$\|f_n-f_m\|_{\infty}\leq\|f_n-f_m\|_{W^{d,1}}\to0,\quad n,m\to\infty.$$
Thus, $\{f_n\}_{n\geq0}$ is a Cauchy sequence in $L_{\infty}(\mathbb{R}^d).$ Therefore, $f_n\to h$ in $L_{\infty}(\mathbb{R}^d)$ (and, therefore, in distributional sense). By uniqueness of the limit, $h=f.$ Hence, $f_n\to f$ in $L_{\infty}(\mathbb{R}^d).$ Since each $f_n$ is continuous, then so is $f.$

{\bf Step 3:} To see the assertion in general case, take $g\in X^{\ast}.$ The function $l_g:s\to \langle g,f(s)\rangle$ belongs to $W^{d,1}(\mathbb{R}^d).$ By Step 2, $l_g$ is continuous for every $g\in X^{\ast}$ and, therefore, $f$ is weak$^{\ast}$ continuous. Clearly,
$$\|l\|_{W^{d,1}}\leq\|g\|_{X^{\ast}}\|f\|_{W^{d,1}(\mathbb{R}^d,X)}.$$
By Step 2, we have
$$\|l\|_{\infty}\leq\|g\|_{X^{\ast}}\|f\|_{W^{d,1}(\mathbb{R}^d,X)}.$$
Taking supremum over the unit ball in $X^{\ast},$ we obtain
$$\|f\|_{L_{\infty}(\mathbb{R}^d,X)}\leq\|f\|_{W^{d,1}(\mathbb{R}^d,X)}.$$
\end{proof}

\begin{lem}\label{poisson second lemma} We have $\mathcal{F}((L_1\cap L_{\infty})(\mathbb{R}^d,X))\subset BWC(\mathbb{R}^d,X).$ 
\end{lem}
\begin{proof} Let $f\in L_1(\mathbb{R}^d,X).$ It is obvious that
$$\|\mathcal{F}f\|_{L_{\infty}(\mathbb{R}^d,X)}\leq\|f\|_{L_1(\mathbb{R}^d,X)}.$$

Let $\mathbb{B}^d$ be the unit ball in $\mathbb{R}^d$ centered at $0.$

Fix $\epsilon>0$ and choose $n\in\mathbb{N}$ such that
$$\int_{\mathbb{R}^d\backslash n\mathbb{B}^d}\|f(u)\|_Xdu<\epsilon.$$
We have
$$\Big\|(\mathcal{F}f)(s_1)-(\mathcal{F}f)(s_2)\Big\|_X\leq\Big\|\int_{\mathbb{R}^d\backslash n\mathbb{B}^d}f(u)(e^{-2\pi i\langle u,s_1\rangle}-e^{-2\pi i\langle u,s_1\rangle})du\Big\|_X+$$
$$+\Big\|\int_{n\mathbb{B}^d}f(u)(e^{-2\pi i\langle u,s_1\rangle}-e^{-2\pi i\langle u,s_1\rangle})du\Big\|_X\leq$$
$$\leq 2\epsilon+\|f\|_{L_{\infty}(\mathbb{R}^d,X)}\cdot \int_{n\mathbb{B}^d}|e^{-2\pi i\langle u,s_1\rangle}-e^{-2\pi i\langle u,s_1\rangle}|du\leq$$
$$\leq 2\epsilon+ 2\pi\|f\|_{L_{\infty}(\mathbb{R}^d,X)}\cdot \int_{n\mathbb{B}^d}\|u\|_2\|s_1-s_2\|_2du\leq$$
$$\leq 2\epsilon+2\pi\|f\|_{L_{\infty}(\mathbb{R}^d,X)}\cdot n^{d+1}\|s_1-s_2\|_2\cdot{\rm vol}(\mathbb{B}_d).$$
If
$$\|s_1-s_2\|_2\leq n^{-d-1}\|f\|_{L_{\infty}(\mathbb{R}^d,X)}^{-1}\epsilon,$$
then
$$\|(\mathcal{F}f)(s_1)-(\mathcal{F}f)(s_2)\|_X\leq c_d\epsilon.$$
Since $\epsilon$ is arbitrarily small, the assertion follows.
\end{proof}

\begin{lem}\label{poisson third lemma} We have $W^{d,1}(\mathbb{R}^d,X)\subset (l_1(L_{\infty}))(\mathbb{R}^d,X).$  Moreover, we have
$$\|f\|_{l_1(L_{\infty})(\mathbb{R}^d,X)}\leq c_d\|f\|_{W^{d,1}(\mathbb{R}^d,X)},\quad f\in W^{d,1}(\mathbb{R}^d,X).$$
\end{lem}
\begin{proof} Let $\phi$ be a smooth function supported on $[-1,1]^d$ such that $\phi=1$ on $[-\frac12,\frac12]^d.$ We have
$$\|f\|_{(l_1(L_{\infty}))(\mathbb{R}^d,X)}\leq\sum_{n\in\mathbb{Z}^d}\|f\cdot T_n\phi\|_{L_{\infty}(\mathbb{R}^d,X)}.$$
Using Lemma \ref{poisson first lemma}, we have
$$\|f\|_{(l_1(L_{\infty}))(\mathbb{R}^d,X)}\leq\sum_{n\in\mathbb{Z}^d}\|f\cdot T_n\phi\|_{W^{d,1}(\mathbb{R}^d,X)}\leq$$
$$\leq\|\phi\|_{C^d([-1,1]^d)}\cdot\sum_{n\in\mathbb{Z}^d}\|f\|_{W^{d,1}(n+[-1,1]^d,X)}=2^d\|\phi\|_{C^d([-1,1]^d)}\|f\|_{W^{d,1}(\mathbb{R}^d,X)}.$$
\end{proof}

\begin{lem}\label{poisson fourth lemma} For $p>d,$ we have $\mathcal{F}(W^{p,1}(\mathbb{R}^d,X))\subset (l_1(L_{\infty}))(\mathbb{R}^d,X).$ Moreover, we have
$$\|\mathcal{F}f\|_{l_1(L_{\infty})(\mathbb{R}^d,X)}\leq c_{p,d}\|f\|_{W^{p,1}(\mathbb{R}^d,X)},\quad f\in W^{p,1}(\mathbb{R}^d,X).$$
\end{lem}
\begin{proof} We have
$$|s|^p\cdot \|(\mathcal{F}f)(s)\|_X=\|(\mathcal{F}(\Delta^{\frac{p}{2}}f))(s)\|_X\leq\|\mathcal{F}(\Delta^{\frac{p}{2}}f)\|_{L_{\infty}(\mathbb{R}^d,X)}\leq$$
$$\leq\|\Delta^{\frac{p}{2}}f\|_{L_1(\mathbb{R}^d,X)}\leq\|f\|_{W^{p,1}(\mathbb{R}^d,X)},$$
$$\|(\mathcal{F}f)(s)\|_X\leq\|\mathcal{F}(f)\|_{L_{\infty}(\mathbb{R}^d,X)}\leq \|f\|_{L_1(\mathbb{R}^d,X)}\leq\|f\|_{W^{p,1}(\mathbb{R}^d,X)}.$$
That is,
$$\|(\mathcal{F}f)(s)\|_X\leq\min\{|s|^{-p},1\}\cdot \|f\|_{W^{p,1}(\mathbb{R}^d,X)}.$$
Since the mapping
$$s\to \min\{|s|^{-p},1\}$$
belongs to $(l_1(L_{\infty}))(\mathbb{R}^d),$ the assertion follows.
\end{proof}

\begin{proof}[Proof of Theorem \ref{poisson thm}] By Lemma \ref{poisson first lemma}, $f(s)$ makes sense for every $f\in W^{d,1}(\mathbb{R}^d,X)$ and for every $s\in\mathbb{R}^d.$ By Lemma \ref{poisson third lemma}, we have
$$\sum_{n\in\mathbb{Z}^d}\|f(n)\|_X\leq\sum_{n\in\mathbb{Z}^d}\sup_{s\in n+[-\frac12,\frac12]^d}\|f(s)\|_X=\|f\|_{l_1(L_{\infty})(\mathbb{R}^d,X)}\leq c_d\|f\|_{W^{d,1}(\mathbb{R}^d,X)}.$$
In particular, the series in the left hand side converges in $X$ and
$$\|\sum_{n\in\mathbb{Z}^d}f(n)\|_X\leq c_d\|f\|_{W^{d,1}(\mathbb{R}^d,X)}.$$
That is, left hand side defines a bounded mapping $T:W^{d,1}(\mathbb{R}^d)\to X.$

By Lemma \ref{poisson second lemma}, $(\mathcal{F}f)(s)$ makes sense for every $f\in W^{d,1}(\mathbb{R}^d,X)$ and for every $s\in\mathbb{R}^d.$ By Lemma \ref{poisson fourth lemma}, we have
$$\sum_{n\in\mathbb{Z}^d}\|(\mathcal{F}f)(n)\|_X\leq\sum_{n\in\mathbb{Z}^d}\sup_{s\in n+[-\frac12,\frac12]^d}\|(\mathcal{F}f)(s)\|_X=$$
$$=\|\mathcal{F}f\|_{l_1(L_{\infty})(\mathbb{R}^d,X)}\leq c_{p,d}\|f\|_{W^{p,1}(\mathbb{R}^d,X)}.$$
In particular, the series in the right hand side converges in $X$ and
$$\|\sum_{n\in\mathbb{Z}^d}(\mathcal{F}f)(n)\|_X\leq c_{p,d}\|f\|_{W^{p,1}(\mathbb{R}^d,X)}.$$
That is, right hand side defines a bounded mapping $S:W^{p,1}(\mathbb{R}^d)\to X.$

If $f$ is vector valued Schwartz function and if $g\in X^{\ast},$ then $l_g:s\to\langle g,f(s)\rangle$ is a Schwartz function. We have
$$\langle g,Tf\rangle=\sum_{n\in\mathbb{Z}^d}\langle g,f(n)\rangle=\sum_{n\in\mathbb{Z}^d}l_g(n).$$
We also have
$$\langle g,Sf\rangle=\sum_{n\in\mathbb{Z}^d}\langle g,(\mathcal{F}f)(n)\rangle=\sum_{n\in\mathbb{Z}^d}(\mathcal{F}l_g)(n).$$
We take Poisson formula for scalar valued Schwartz functions for granted --- it follows from the distributional equality
$$\sum_{n\in\mathbb{Z}^d}\delta(x-n)=\sum_{n\in\mathbb{Z}^d}e^{2\pi i\langle n,x\rangle},\quad x\in\mathbb{R}^d.$$
That is, we have
$$\sum_{n\in\mathbb{Z}^d}l_g(n)=\sum_{n\in\mathbb{Z}^d}(\mathcal{F}l_g)(n).$$
Combining these $3$ equalities, we infer
$$\langle g,Tf\rangle=\langle g,Sf\rangle,\quad g\in X^{\ast}.$$
In other words, $Tf=Sf$ for every vector valued Schwartz function.

That is, we have $2$ bounded linear maps from $W^{p,1}(\mathbb{R}^d,X)$ to $X.$ These maps coincide on the subspace of vector valued Schwartz functions. Since vector valued Schwartz functions are dense in $W^{p,1}(\mathbb{R}^d,X),$ it follows immediately that these maps coincide on $W^{p,1}(\mathbb{R}^d,X).$ This completes the proof.
\end{proof}

\begin{proof}[Proof of Theorem \ref{sobolev connes lemma}] By Theorem \ref{poisson thm}, we have
$$\sum_{n\in\mathbb{Z}^d}f(tn)=\sum_{n\in\mathbb{Z}^d}(\sigma_{\frac1t}f)(n)=\sum_{n\in\mathbb{Z}^d}(\mathcal{F}\sigma_{\frac1t}f)(n)=t^{-d}\sum_{n\in\mathbb{Z}^d}(\mathcal{F}f)(t^{-1}n).$$
In other words, we have
$$\sum_{n\in\mathbb{Z}^d}f(tn)-t^{-d}\int_{\mathbb{R}^d}f(u)du=t^{-d}\sum_{0\neq n\in\mathbb{Z}^d}(\mathcal{F}f)(t^{-1}n).$$

Recall that
$$|s|^p\cdot \|(\mathcal{F}f)(s)\|_X=\|(\mathcal{F}(\Delta^{\frac{p}{2}}f))(s)\|_X\leq\|\mathcal{F}(\Delta^{\frac{p}{2}}f)\|_{L_{\infty}(\mathbb{R}^d,X)}\leq$$
$$\leq\|\Delta^{\frac{p}{2}}f\|_{L_1(\mathbb{R}^d,X)}\leq\|f\|_{W^{p,1}(\mathbb{R}^d,X)}.$$
Hence, for $n\neq0,$
$$\|(\mathcal{F}f)(t^{-1}n)\|_X\leq \frac{t^p}{|n|^p}\|f\|_{W^{p,1}(\mathbb{R}^d,X)}.$$
We now infer that
$$\Big\|\sum_{0\neq n\in\mathbb{Z}^d}(\mathcal{F}f)(t^{-1}n)\Big\|_X\leq\sum_{0\neq n\in\mathbb{Z}^d}\frac{t^p}{|n|^p}\|f\|_{W^{p,1}(\mathbb{R}^d,X)}=c_pt^p\|f\|_{W^{p,1}(\mathbb{R}^d,X)}.$$
This completes the proof.
\end{proof}

\section{Proof of Theorem \ref{finite expansion thm}}\label{verification section}

\begin{lem}\label{complex shift lemma} We have
$${\rm Good}_k(s)=\frac1{2\pi i}\int_{1+i\mathbb{R}}{\rm good}_k(s,z)e^zdz.$$
\end{lem}
\begin{proof} The crucial fact is that, for $s\neq0,$ the mapping $z\to {\rm good}_k(s,z)$ is holomorphic in the half-plane $\{\Re(z)>-\epsilon\},$ where $\epsilon$ depends on $s.$ We have
\begin{equation}\label{corrk eq}
\frac1{2\pi}\int_{\mathbb{R}}{\rm good}_k(s,i\lambda)e^{i\lambda}d\lambda=\frac1{2\pi i}\int_{i\mathbb{R}}{\rm good}_k(s,z)e^zdz.
\end{equation}

We claim that
$$\int_{1+i\mathbb{R}}{\rm good}_k(s,z)e^zdz=\int_{i\mathbb{R}}{\rm good}_k(s,z)e^zdz.$$
Indeed, we have
$$\int_{1+i\mathbb{R}}{\rm good}_k(s,z)e^zdz=\lim_{N\to\infty}\int_{1-iN}^{1+iN}{\rm good}_k(s,z)e^zdz.$$
Using analyticity and Cauchy theorem, we write
$$\int_{1-iN}^{1+iN}{\rm good}_k(s,z)e^zdz=\int_{1-iN}^{-iN}{\rm good}_k(s,z)e^zdz+$$
$$+\int_{-iN}^{iN}{\rm good}_k(s,z)e^zdz+\int_{iN}^{1+iN}{\rm good}_k(s,z)e^zdz.$$
Obviously,
$$\Big\|\int_{1-iN}^{-iN}{\rm good}_k(s,z)e^zdz\Big\|_{\infty}\leq e\cdot\sup_{t\in(0,1)}\|{\rm good}_k(s,t-iN)\|_{\infty}.$$
Thus,
$$\int_{1-iN}^{-iN}{\rm good}_k(s,z)e^zdz=o(1),\quad N\to\infty.$$
Similarly,
$$\int_{iN}^{1+iN}{\rm good}_k(s,z)e^zdz=o(1),\quad N\to\infty.$$
This proves the claim and, hence, the assertion of the lemma.
\end{proof}

\begin{lem}\label{first monomial estimate} Let $\alpha,\beta\in\mathbb{Z}_+^d$ and let
$$g(s,z)=(x(s)+z)^{-1}s^{\beta}.$$
For $\Re z>0$ and for every $s\in\mathbb{R}^d,$ we have
$$\Big\|\partial^{\alpha}\frac{d^N}{d^Nz}g(s,z)\Big\|_{\infty}=O((|s|^2+|z|)^{\frac12|\beta|_1-\frac12|\alpha|_1-N-1}).$$
\end{lem}
\begin{proof} Set $h(s,z)=(x(s)+z)^{-N}.$ We claim that
$$\Big\|\partial^{\alpha}h(s,z)\Big\|_{\infty}=O((|s|^2+|z|)^{-\frac12|\alpha|_1-N}).$$
We prove the assertion by induction on $N.$ For $N=1,$ it is obvious. Let us prove it for $N+1.$

Let $h=h_1h_2,$ where $h_1(s,z)=(x(s)+z)^{-N}$ and $h_2(s,z)=(x(s)+z)^{-1}.$ By Leibniz rule, we have
$$\partial^{\alpha}h=\sum_{\alpha_1+\alpha_2=\alpha}c^{\alpha}_{\alpha_1}\partial^{\alpha_1}h_1\cdot \partial^{\alpha_2}h_2.$$
By triangle inequality, we have
$$\Big\|\partial^{\alpha}h(s,z)\Big\|_{\infty}\leq\sum_{\alpha_1+\alpha_2=\alpha}c^{\alpha}_{\alpha_1}\Big\|\partial^{\alpha_1}h_1(s,z)\Big\|_{\infty}\cdot\Big\|\partial^{\alpha_2}h_2(s,z)\Big\|_{\infty}.$$
By inductive assumption, we have
$$\Big\|\partial^{\alpha_1}h_1(s,z)\Big\|_{\infty}=O((|s|^2+|z|)^{-\frac12|\alpha_1|_1-N}).$$
Obviously,
$$\Big\|\partial^{\alpha_2}h_2(s,z)\Big\|_{\infty}=O((|s|^2+|z|)^{-\frac12|\alpha_2|_1-1}).$$
Combining these 3 estimates we establish the claim.

The assertion for $\beta=0$ follows immediately from the claim above.

Consider now the general case. Let $g=h_1h_2,$ where $h_1(s,z)=s^{\beta}$ and $h_2(s,z)=(x(s)+z)^{-1}.$ By Leibniz rule, we have
$$\partial^{\alpha}\frac{d^N}{d^Nz}g=\sum_{\alpha_1+\alpha_2=\alpha}c^{\alpha}_{\alpha_1}\partial^{\alpha_1}h_1\cdot \partial^{\alpha_2}\frac{d^N}{d^Nz}h_2.$$
By triangle inequality, we have
$$\Big\|\partial^{\alpha}\frac{d^N}{d^Nz}g(s,z)\Big\|_{\infty}\leq\sum_{\alpha_1+\alpha_2=\alpha}c^{\alpha}_{\alpha_1}\Big\|\partial^{\alpha_1}h_1(s,z)\Big\|_{\infty}\cdot\Big\|\partial^{\alpha_2}\frac{d^N}{d^Nz}h_2(s,z)\Big\|_{\infty}.$$
By the special case proved above, we have
$$\Big\|\partial^{\alpha_2}\frac{d^N}{d^Nz}h_2(s,z)\Big\|_{\infty}=O((|s|^2+|z|)^{-\frac12|\alpha_2|_1-N-1}).$$
Evidently,
$$\Big\|\partial^{\alpha_1}h_1(s,z)\Big\|_{\infty}=
\begin{cases}
O(|s|^{|\beta|_1-|\alpha|_1})&\alpha\leq\beta\\
0,&\alpha\not\leq\beta
\end{cases}.$$
A combination of these 3 estimates yields the assertion.
\end{proof}

\begin{lem}\label{second monomial estimate} Let $a_l\in L_{\infty}(\mathbb{T}^d_{\theta}),$ $1\leq l\leq L,$ and let
$$g(s,z)=\prod_{l=1}^La_l(x(s)+z)^{-1}.$$
For $\Re z>0$ and for every $s\in\mathbb{R}^d,$ we have
$$\Big\|\partial^{\alpha}\frac{d^N}{d^Nz}g(s,z)\Big\|_{\infty}=O((|s|^2+|z|)^{-\frac12|\alpha|_1-L-N}).$$
\end{lem}
\begin{proof} We prove the assertion by induction on $L.$ For $L=1,$ the assertion follows from Lemma \ref{first monomial estimate} (applied with $\beta=0$). Suppose the assertion holds for $L.$ Let us prove it for $L+1.$

Let $g=h_1h_2,$ where
$$h_1(s,z)=\prod_{l=1}^La_l(x(s)+z)^{-1},\quad h_2(s,z)=a_{L+1}(x(s)+z)^{-1}.$$
By Leibniz rule, we have
$$\partial^{\alpha}\frac{d^N}{d^Nz}g=\sum_{\substack{\alpha_1+\alpha_2=\alpha\\ N_1+N_2=N}}c^{N,\alpha}_{N_1,\alpha_1}\partial^{\alpha_1}\frac{d^{N_1}}{d^{N_1}z}h_1\cdot \partial^{\alpha_2}\frac{d^{N_2}}{d^{N_2}z}h_2.$$
By triangle inequality, we have
$$\Big\|\partial^{\alpha}\frac{d^N}{d^Nz}g(s,z)\Big\|_{\infty}\leq\sum_{\substack{\alpha_1+\alpha_2=\alpha\\ N_1+N_2=N}}c^{N,\alpha}_{N_1,\alpha_1}\Big\|\partial^{\alpha_1}\frac{d^{N_1}}{d^{N_1}z}h_1(s,z)\Big\|_{\infty}\cdot\Big\|\partial^{\alpha_2}\frac{d^{N_2}}{d^{N_2}z}h_2(s,z)\Big\|_{\infty}.$$
By inductive assumption, we have
$$\Big\|\partial^{\alpha_1}\frac{d^{N_1}}{d^{N_1}z}h_1(s,z)\Big\|_{\infty}=O((|s|^2+|z|)^{-\frac12|\alpha_1|_1-L-N_1}).$$
By Lemma \ref{first monomial estimate}, we have
$$\Big\|\partial^{\alpha_2}\frac{d^{N_2}}{d^{N_2}z}h_2(s,z)\Big\|_{\infty}=O((|s|^2+|z|)^{-\frac12|\alpha_2|_1-1-N_2}).$$
A combination of these 3 estimates yields the assertion.
\end{proof}

\begin{lem}\label{third monomial estimate} Let $a_l\in L_{\infty}(\mathbb{T}^d_{\theta}),$ $1\leq l\leq L,$ and let
$$g(s,z)=(x(s)+z)^{-1}s^{\beta}\prod_{l=1}^La_l(x(s)+z)^{-1}.$$
For $\Re z>0$ and for every $s\in\mathbb{R}^d,$ we have
$$\Big\|\partial^{\alpha}\frac{d^N}{d^Nz}g(s,z)\Big\|_{\infty}=O((|s|^2+|z|)^{\frac12|\beta|_1-\frac12|\alpha|_1-L-N-1}).$$
\end{lem}
\begin{proof} Let
$$h_1(s,z)=(x(s)+z)^{-1}s^{\beta},\quad h_2(s,z)=\prod_{l=1}^La_l(x(s)+z)^{-1}.$$
By Leibniz rule, we have
$$\partial^{\alpha}\frac{d^N}{d^Nz}g=\sum_{\substack{\alpha_1+\alpha_2=\alpha\\ N_1+N_2=N}}c^{N,\alpha}_{N_1,\alpha_1}\partial^{\alpha_1}\frac{d^{N_1}}{d^{N_1}z}h_1\cdot \partial^{\alpha_2}\frac{d^{N_2}}{d^{N_2}z}h_2.$$
By triangle inequality, we have
$$\Big\|\partial^{\alpha}\frac{d^N}{d^Nz}g(s,z)\Big\|_{\infty}\leq\sum_{\substack{\alpha_1+\alpha_2=\alpha\\ N_1+N_2=N}}c^{N,\alpha}_{N_1,\alpha_1}\Big\|\partial^{\alpha_1}\frac{d^{N_1}}{d^{N_1}z}h_1(s,z)\Big\|_{\infty}\cdot\Big\|\partial^{\alpha_2}\frac{d^{N_2}}{d^{N_2}z}h_2(s,z)\Big\|_{\infty}.$$
By Lemma \ref{first monomial estimate}, we have
$$\Big\|\partial^{\alpha_1}\frac{d^{N_1}}{d^{N_1}z}h_1(s,z)\Big\|_{\infty}=O((|s|^2+|z|)^{\frac12|\beta|_1-\frac12|\alpha_1|_1-N_1-1}).$$
By Lemma \ref{second monomial estimate}, we have
$$\Big\|\partial^{\alpha_2}\frac{d^{N_2}}{d^{N_2}z}h_2(s,z)\Big\|_{\infty}=O((|s|^2+|z|)^{-\frac12|\alpha_2|_1-L-N_2}).$$
A combination of these 3 estimates yields the assertion.
\end{proof}

\begin{lem}\label{goodk symbol estimate} For $\Re z>0$ and for every $s\in\mathbb{R}^d,$ we have
$$\Big\|\partial^{\alpha}\frac{d^N}{d^Nz}{\rm good}_k(s,z)\Big\|_2=O((|s|^2+|z|)^{-\frac{k}2-1-\frac12|\alpha|_1-N}).$$
\end{lem}
\begin{proof} By induction, ${\rm good}_k(s,z)$ is a sum of finitely many terms of the shape
$$g(s,z)=(x(s)+z)^{-1}s^{\beta}\cdot \prod_{l=1}^La_l(x(s)+z)^{-1},$$
where $|\beta|_1=2L-k$ and $a_l\in C^{\infty}(\mathbb{T}^d_{\theta}),$ $1\leq l\leq L.$ The assertion follows from Lemma \ref{third monomial estimate}.
\end{proof}

\begin{lem}\label{main verification lemma} For $k\geq0$ and $p>0,$ we have ${\rm Good}_k\in W^{p,1}(\mathbb{R}^d,L_2(\mathbb{T}^d_{\theta})).$
\end{lem}
\begin{proof} Using Lemma \ref{complex shift lemma} and integration by parts, we obtain
$${\rm Good}_k(s)=\frac{(-1)^N}{2\pi i}\int_{1+i\mathbb{R}}\frac{d^N}{d^Nz}{\rm good}_k(s,z)e^zdz.$$
Heuristically, we have
$$\partial^{\alpha}{\rm Good}_k(s)=\frac{(-1)^N}{2\pi i}\int_{1+i\mathbb{R}}\partial^{\alpha}\frac{d^N}{d^Nz}{\rm good}_k(s,z)e^zdz.$$
This formula is indeed true because the integral in the right hand side converges absolutely by Lemma \ref{goodk symbol estimate}. Moreover, we have
$$\Big\|\partial^{\alpha}{\rm Good}_k(s)\Big\|_2\leq\frac{e}{2\pi}\int_{1+i\mathbb{R}}\Big\|\partial^{\alpha}\frac{d^N}{d^Nz}{\rm good}_k(s,z)\Big\|_2|dz|\leq$$
$$\leq c_{N,\alpha,g}\int_{\mathbb{R}}(|s|^2+1+|\lambda|)^{-\frac{k}2-1-\frac12|\alpha|_1-N}d\lambda\leq$$
$$\leq c'_{N,\alpha,g}(|s|^2+1)^{-\frac{k}2-\frac12|\alpha|_1-N}.$$
In particular, $\partial^{\alpha}{\rm Good}_k\in L_1(\mathbb{R}^d,L_2(\mathbb{T}^d_{\theta}))$ for every $\alpha\in\mathbb{Z}^d_+.$
\end{proof}

\begin{cor}\label{series convergence corollary} For every $k\geq0,$ the series
$$\sum_{n\in\mathbb{Z}^d}{\rm Good}_k(nt^{\frac12})$$
converges in $L_2(\mathbb{T}^d_{\theta}).$ We have
$$\sum_{n\in\mathbb{Z}^d}{\rm Good}_k(nt^{\frac12})=t^{-\frac{d}{2}}\cdot\int_{\mathbb{R}^d}{\rm Good}_k(s)ds+O(t^{\infty}).$$
\end{cor}
\begin{proof} The assertion follows immediately from Lemma \ref{main verification lemma} and Theorem \ref{sobolev connes lemma}.
\end{proof}

\begin{proof}[Proof of Theorem \ref{finite expansion thm}] Let
$$F(t)=\sum_{n\in\mathbb{Z}^d}(\lambda_r(e_n)^*e^{-tA_g}\lambda_r(e_n))(1),$$
where the series converges weakly in $L_2(\mathbb{T}^d_{\theta})$ by Lemma \ref{first strategy lemma}. By Theorem \ref{exponent splitting theorem}, we have
$$\Big\|F(t)-\sum_{n\in\mathbb{Z}^d}\Big(\sum_{k=0}^{2d}t^{\frac{k}{2}}{\rm Good}_k(nt^{\frac12})\Big)\Big\|_2=O(1).$$
By Corollary \ref{series convergence corollary}, we have
$$\Big\|F(t)-\sum_{k=0}^{2d}t^{\frac{k-d}{2}}\int_{\mathbb{R}^d}{\rm Good}_k(s)ds\Big\|_2=O(1).$$

Obviously, the terms with $k\geq d$ are bounded. Since ${\rm Good}_k$ is an odd function when $k$ is odd, it follows that respective summand is $0.$ Recall that
$$I_k=\int_{\mathbb{R}^d}{\rm Good}_k(s)ds,\quad 0\leq k\leq d.$$
We now have
$$\Big\|F(t)-\sum_{\substack{0\leq k<d\\ k=0{\rm mod}2}}t^{\frac{k-d}{2}}I_k\Big\|_2=O(1).$$
The assertion follows now from Lemma \ref{first strategy lemma}.
\end{proof}

\section{Proof of the main result}

Take $d'>d$ and consider $d'\times d'$ matrix $\theta'$ whose left upper corner is $\theta.$ For simplicity,  it makes sense to set $\theta'_{kl}=0$ when $k>d$ or when $l>d.$ We have $L_{\infty}(\mathbb{T}^{d'}_{\theta'})=L_{\infty}(\mathbb{T}^d_{\theta})\bar{\otimes}L_{\infty}(\mathbb{T}^{d'-d})$ (see Example \ref{tensor example}). 

Define a metric $g'$ (size of $g'$ is $d'$) whose left upper corner is $g.$ We ask that $g_{kl}=\delta_{k,l}$ when either $k>d$ or $l>d.$ 

\begin{lem}\label{nu ext lemma} Let $\nu'$ be a version of $\nu$ constructed from the metric tensor $g'.$ We have
$$\nu'=\nu\otimes 1.$$
\end{lem}
\begin{proof} We have
$$e^{-\sum_{i,j=1}^{d'}t_it_j((g')^{-1})_{ij}}=e^{-\sum_{i,j=1}^dt_it_j(g^{-1})_{ij}}\otimes 1$$
$$\nu'=\pi^{-\frac{d}{2}}\int_{\mathbb{R}^d}e^{-\sum_{i,j=1}^dt_it_j(g^{-1})_{ij}}dt\otimes \pi^{-\frac{d'-d}{2}}\int_{\mathbb{R}^{d'-d}}e^{-\sum_{i=d+1}^{d'}t_i^2}dt=\nu\otimes 1.$$
\end{proof}

\begin{lem}\label{corrk ext lemma} Let ${\rm Corr}_k'$ be a version of ${\rm Corr}_k$ constructed from the metric tensor $g'.$ We have
$${\rm Corr}_k'(s)={\rm Corr}_k(u)\otimes e^{-|v|^2},\quad u=(s_1,\cdots,s_d),\quad v=(s_{d+1},\cdots,s_{d'}).$$
\end{lem}
\begin{proof} Let ${\rm corr}_k'$ be a version of ${\rm corr}_k$ constructed from the metric tensor $g'.$ 

We have
$${\rm corr}_k'(s,i\lambda)={\rm corr}_k(u,|v|^2+i\lambda)\otimes 1,\quad u=(s_1,\cdots,s_d),\quad v=(s_{d+1},\cdots,s_{d'}).$$
The crucial fact is that, for $u\neq0,$ the mapping $z\to {\rm corr}_k(u,z)$ is holomorphic in the half-plane $\{\Re(z)>-\epsilon\},$ where $\epsilon$ depends on $u.$ Therefore, we have
\begin{equation}\label{corrk' eq}
\frac1{2\pi}\int_{\mathbb{R}}{\rm corr}_k'(s,i\lambda)e^{i\lambda}d\lambda=e^{-|v|^2}\cdot\frac1{2\pi i}\int_{|v|^2+i\mathbb{R}}{\rm corr}_k(u,z)e^zdz.
\end{equation}

We claim that
$$\int_{|v|^2+i\mathbb{R}}{\rm corr}_k(u,z)e^zdz=\int_{i\mathbb{R}}{\rm corr}_k(u,z)e^zdz.$$
Indeed, we have
$$\int_{|v|^2+i\mathbb{R}}{\rm corr}_k(u,z)e^zdz=\lim_{N\to\infty}\int_{|v|^2-iN}^{|v|^2+iN}{\rm corr}_k(u,z)e^zdz.$$
We now write
$$\int_{|v|^2-iN}^{|v|^2+iN}{\rm corr}_k(u,z)e^zdz=\int_{|v|^2-iN}^{-iN}{\rm corr}_k(u,z)e^zdz+$$
$$+\int_{-iN}^{iN}{\rm corr}_k(u,z)e^zdz+\int_{iN}^{|v|^2+iN}{\rm corr}_k(u,z)e^zdz.$$
Obviously,
$$\Big\|\int_{|v|^2-iN}^{-iN}{\rm corr}_k(u,z)e^zdz\Big\|_{\infty}\leq|v|^2e^{|v|^2}\cdot\sup_{t\in(0,|v|^2)}\|{\rm corr}_k(u,t-iN)\|_{\infty}.$$
Thus,
$$\int_{|v|^2-iN}^{-iN}{\rm corr}_k(u,z)e^zdz=o(1),\quad N\to\infty.$$
Similarly,
$$\int_{iN}^{|v|^2+iN}{\rm corr}_k(u,z)e^zdz=o(1),\quad N\to\infty.$$
This proves the claim.

The assertion follows from the above claim and \eqref{corrk' eq}.
\end{proof}

\begin{lem}\label{ik ext lemma} Let $I_k'$ be a version of $I_k$ constructed from the metric tensor $g'.$ We have
$$I_k'=\pi^{\frac{d'-d}{2}}I_k\otimes 1.$$
\end{lem}
\begin{proof} Obviously,
$$\int_{\mathbb{R}^{d'-d}}e^{-|v|^2}dv=\pi^{\frac{d'-d}{2}}.$$
The assertion follows now from Lemma \ref{corrk ext lemma}.
\end{proof}

\begin{proof}[Proof of Theorem \ref{main result}] By Lemma \ref{nu ext lemma}, $\lambda_l(\nu')$ commutes with $D_k,$ $k>d.$ Therefore, we have
$$A_{g'}=\lambda_l((\nu')^{-\frac12})\sum_{i,j=1}^dD_i\lambda_l((\nu')^{\frac12}(g^{-1})_{ij}(\nu')^{\frac12})D_j\lambda_l((\nu')^{-\frac12})+\sum_{i=d+1}^{d'}D_i^2.$$
By Lemma \ref{nu ext lemma}, the first summand is exactly $A_g\otimes 1.$ The second summand is, clearly, $1\otimes \Delta.$ Consequently, we have
$$A_{g'}=A_g\otimes 1+1\otimes \Delta.$$
This implies
$$e^{-tA_{g'}}=e^{-tA_g\otimes 1-1\otimes t\Delta}=e^{-tA_g}\otimes e^{-t\Delta}.$$

Also, if $x\in L_{\infty}(\mathbb{T}^d_{\theta}),$ then $x\otimes 1\in L_{\infty}(\mathbb{T}^{d'}_{\theta'}).$ We have
$$\lambda_l(x\otimes 1)e^{-tA_{g'}}=\lambda_l(x)e^{-tA_g}\otimes e^{-t\Delta}.$$
Therefore,
$${\rm Tr}(\lambda_l(x\otimes 1)e^{-tA_{g'}})={\rm Tr}(\lambda_l(x)e^{-tA_g})\cdot{\rm Tr}(e^{-t\Delta}).$$

It follows from the Poisson summation formula that
$${\rm Tr}(e^{-t\Delta})=(\frac{\pi}{t})^{\frac{d'-d}{2}}\cdot\Big(1+O(t^{\infty})\Big).$$
By Theorem \ref{finite expansion thm}, we have 
$${\rm Tr}(\lambda_l(w)e^{-tA_{g'}})=t^{-\frac{d'}{2}}\sum_{\substack{0\leq k<d'\\ k=0{\rm mod}2}}t^{\frac{k}{2}}\tau(w I_k')+O(1),\quad t\downarrow 0,$$
for every $w\in L_{\infty}(\mathbb{T}^{d'}_{\theta'}).$ Setting $w=x\otimes 1,$ we infer from Lemma \ref{ik ext lemma} that
$${\rm Tr}(\lambda_l(x)e^{-tA_g})\cdot (\frac{\pi}{t})^{\frac{d'-d}{2}}\cdot\Big(1+O(t^{\infty})\Big)=\pi^{\frac{d'-d}{2}}t^{-\frac{d'}{2}}\sum_{\substack{0\leq k<d'\\ k=0{\rm mod}2}}t^{\frac{k}{2}}\tau(xI_k)+O(1).$$
It follows immediately that
$${\rm Tr}(\lambda_l(x)e^{-tA_g})=t^{-\frac{d}{2}}\sum_{\substack{0\leq k<d'\\ k=0{\rm mod}2}}t^{\frac{k}{2}}\tau(xI_k)+O(t^{\frac{d'-d}{2}}).$$

Taking as large $d'$ as needed, we obtain an asymptotic expansion.

Finally, we have
$$e^{-t\Delta_g}=\lambda_l(\nu^{-\frac12})e^{-tA_g}\lambda_l(\nu^{\frac12}).$$
Thus,
$${\rm Tr}(\lambda_l(x)e^{-t\Delta_g})={\rm Tr}(\lambda_l(x)\cdot \lambda_l(\nu^{-\frac12})e^{-tA_g}\lambda_l(\nu^{\frac12})={\rm Tr}(\lambda_l(\nu^{\frac12}x\nu^{-\frac12})e^{-tA_g}).$$
By the already proved asymptotic expansion, we have
$${\rm Tr}(\lambda_l(x)e^{-t\Delta_g})\sim t^{-\frac{d}{2}}\sum_{k=0{\rm mod}2}t^{\frac{k}{2}}\tau(\nu^{\frac12}x\nu^{-\frac12}\cdot I_k)=t^{-\frac{d}{2}}\sum_{k=0{\rm mod}2}t^{\frac{k}{2}}\tau(x\cdot \nu^{-\frac12}I_k\nu^{\frac12}).$$
\end{proof}

\end{document}